\documentclass{amsart}
\usepackage{amssymb,latexsym,amsmath,amsthm}
\usepackage[mathscr]{eucal}

\newtheorem{theorem}{Theorem}[section]

\newtheorem{proposition}[theorem]{Proposition}

\theoremstyle{definition}

\numberwithin{equation}{section}

\renewcommand{\l}{\lambda}

\newcommand{\g}{{\rm g}}
\newcommand{\RR}{\ensuremath{\mathbb{R}}}
\newcommand{\CC}{\ensuremath{\mathbb{C}}}
\newcommand{\sph}{\ensuremath{\mathbb{S}}}
\newcommand{\prtl}{\ensuremath{\partial}}

\newcommand{\hf}{\ensuremath{\frac{1}{2}}}
\newcommand{\supp}{\ensuremath{\text{supp}}}
\newcommand{\veps}{\ensuremath{\varepsilon}}

\title[Refined local smoothing for the Schr\"odinger equation in domains]{On refined local smoothing estimates for the Schr\"odinger equation in exterior domains}
\author{Matthew D. Blair}
\address{Department of Mathematics and Statistics, University of New Mexico,
Albuquerque, NM, 87131} \email{blair@math.unm.edu}
\thanks{The author was supported by the National Science Foundation grants DMS-0801211, DMS-1001529.}
\begin{document}

\begin{abstract}
We consider refinements of the local smoothing estimates for the Schr\"odinger equation in domains which are exterior to a strictly convex obstacle in $\RR^n$.  By restricting the solution to small, frequency dependent collars of the boundary, it is expected that taking its square integral in space-time should exhibit a larger gain in regularity when compared to the usual gain of half a derivative.  By a result of Ivanovici, these refined local smoothing estimates are satisfied by solutions in the exterior of a ball.  We show that when such estimates are valid, they can be combined with wave packet parametrix constructions to yield Strichartz estimates.  This provides an avenue for obtaining these bounds when Neumann boundary conditions are imposed.
\end{abstract}
\maketitle
\section{Introduction}\label{sec:intro}
Let $(\Omega,\g)$ be a Riemannian manifold with boundary of dimension
$n \geq 2$, and let $u(t,x): [-T,T] \times \Omega \to \CC$ be the solution
to the Schr\"{o}dinger equation
\begin{equation}\label{schrodeqn}
(D_t + \Delta_\g)v(t,x) =0\,, \qquad v(0,x)= f(x)\,,
\end{equation}
where $\Delta_\g$ is assumed to be positive and $D_t = -i \prtl_t$.  We assume in addition that $v$ satisfies either Dirichlet or Neumann boundary conditions
\begin{equation}\label{bcs}
v(t,x)\big|_{x \in \prtl \Omega} =0 \qquad \text{ or } \qquad \prtl_\nu
v(t,x)\big|_{x \in \prtl \Omega} =0\,,
\end{equation}
where $\prtl_\nu$ denotes the normal derivative along the boundary.

In recent years there has been a great deal of interest in establishing space-time integrability estimates for solutions to \eqref{schrodeqn}.  One family of particular interest are the Strichartz estimates, which state that for certain triples $(p,q,s)$ with $2< p \leq \infty,$ $2 \leq q < \infty$, and $s \geq 0$
\begin{equation}\label{str}
\|v\|_{L^p((-T,T);L^q(\Omega))} \lesssim \|f\|_{H^s(\Omega)} .
\end{equation}
Here $H^s(\Omega)$ denotes the $L^2$ Sobolev space of order $s$, defined with respect to the spectral resolution of either the Dirichlet or Neumann Laplacian (cf. the concluding remark in \cite[\S1]{bssschrod}).  Strichartz inequalities provide one of the most efficient ways of handling the perturbative theory for many nonlinear Schr\"odinger equations.  The nonlinearity appearing in these equations often involve powers of the solution on the right hand side, and as such these inequalities such as \eqref{str} provide an effective avenue for controlling the strength of the nonlinearity.

Any solution to \eqref{schrodeqn} in Euclidean space ($\Omega = \RR^n$, $\g^{ij} = \delta_{ij}$) can be rescaled to produce a new solution to the same equation.  This gives rise to the admissability condition on the triple $(p,q,s)$
\begin{equation}\label{admiss}
\frac 2p + \frac nq \geq \frac n2 -s.
\end{equation}
Analogous considerations show that this restriction must also hold for any equation posed on a manifold.  When equality holds in \eqref{admiss}, the estimate is said to be \emph{scale invariant}. Otherwise, we say there is a \emph{loss of derivatives} in the estimate as it deviates from the optimal regularity predicted by scaling.

Strichartz estimates are best understood for the equation posed on Euclidean space see \cite{strich77}, \cite{ginvelo85}, \cite{keeltao98} and references therein.  In this case, the scale invariant estimates hold for any triple with $s = 0$ and one can take $T=\infty$.   Sobolev embedding then implies estimates for any $s>0$.  We therefore refer to exponents $p,q$ satisfying $\frac 2p + \frac nq < \frac n2$ as \emph{subcritical} since the proof of scale invariant estimates in this case does not use the full rate of dispersion for solutions to \eqref{schrodeqn}.  Otherwise if $\frac 2p + \frac nq = \frac n2$, the exponents are considered to be \emph{critical}.

The issue is considerably more difficult when one begins to consider boundary value problems in $\RR^n$.  This is due to several reasons, the most notable of which is that boundary conditions begin to affect the flow of energy, which in turn can inhibit dispersion, complicate parametrix constructions, or both.  In spite of this, there has been some partial progress in this area.  To date, the strongest results are for solutions in nontrapping exterior domains $\Omega = \RR^n\setminus\mathcal{K}$, $\g_{ij} = \delta_{ij}$ where $\mathcal{K}$ is taken to be a compact obstacle whose boundary forms a smooth embedded hypersurface in $\RR^n$.  An exterior domain is said to be \emph{nontrapping} if every unit speed broken bicharacteristic escapes a compact set in $\overline{\Omega}$ in finite time.  In this case, one has the following local smoothing estimate of Burq-G\'erard-Tzvetkov \cite{bgtexterior}
\begin{equation}\label{locsmooth}
\|\phi v\|_{L^2((-T,T); H^{s+\hf}(\Omega))} \leq C\|f\|_{H^s(\Omega)}\,, \qquad
\phi \in C_c^\infty (\overline{\Omega}).
\end{equation}

Local smoothing estimates have a long tradition in the analysis of Schr\"odinger equations on $\RR^n$ and originate in the work of Constantin and
Saut~\cite{consaut}, Sj\"olin~\cite{Sjolin}, Vega~\cite{vegasmooth}, and others.
One heuristic argument for the estimate \eqref{locsmooth} follows by wave packet analysis.  A coherent wave packet supported at a large frequency scale $\l$ should propagate at speed $\approx \l$ and hence spend time $\approx 1/\l$ within the support of $\phi$.  Taking the square integral in time should thus yield a gain of one half a derivative.

The connection between local smoothing bounds and Strichartz estimates was observed by Journ\'e-Soffer-Sogge \cite{jss}, who considered Schr\"odinger equations on $\RR^n$ involving a potential term.  They observed that local smoothing bounds control the error which arises by taking the free evolution to be a parametrix for the equation.  A similar approach was used by Staffilani-Tataru \cite{ST} to establish scale invariant Strichartz estimates in certain nontrapping metric perturbations of the Laplacian on $\RR^n$.  Here the idea is that local smoothing bounds control the errors which arise in localizing the problem in space, which can be accomplished by using smooth cutoff functions.  This in turn reduces matters to establishing a parametrix for the equation which may only invert the equation locally, say within the domain of a suitable local diffeomorphism.  This approach was then adapted to exterior domain problems by Burq-G\'erard-Tzvetkov \cite{bgtexterior} and Anton \cite{antonext}.  However in each case, the parametrix construction involved did not yield scale invariant estimates.

Recently, there have been a few results which have improved the losses coming from these parametrix constructions.  The works of Planchon and Vega \cite{planvega} and the author with Smith and Sogge \cite{bssschrod} prove scale invariant estimates for equations in any nontrapping exterior domain, but have restrictions on the admissibility of the Lebesgue exponents $p,q$.  That is, the estimates are only valid for a subset of the exponents $p,q$ satisfying \eqref{admiss}.  However, when the obstacle $\mathcal{K}$ is strictly convex, Ivanovici \cite{ivanomt} showed that the full range of Strichartz estimates hold for the Dirichlet problem.  She showed that the Melrose-Taylor parametrix inverts the equation locally and then used it to prove the desired estimates.

In this work, we consider an alternative approach to Strichartz estimates in domains exterior to a strictly convex obstacle.  It uses refinements of the local smoothing estimates to control the error terms which arise in the wave packet parametrix construction of \cite{bssschrod}.  Such an approach has already been considered in two contexts.  One is the work of Tataru \cite{tataruschrod}, who once again considered variable coefficient Schr\"odinger equations, but treated a more general family of asymptotically flat metrics.  Here a wave packet parametrix is used, but in order to control the error terms, a local smoothing estimate is needed on frequency dependent scales.  An improved estimate is obtained by restricting the solution to annuli whose size may depend on the frequency of the solution.  A related approach was considered in other works of Ivanovici \cite{ivanoballs}, \cite{ivanorevise}, who proved frequency dependent estimates for domains which are exterior to the unit ball in $\RR^n$ (cf. Theorem \ref{thm:nrgdk} below).  Here the solution is restricted to a frequency dependent collars about the boundary.  It was then shown that when such an estimate is combined with Sobolev embedding, Strichartz estimates with a loss of derivatives follow as a result.

In order to describe the refined local smoothing estimates considered here, we let $\beta$ be a smooth bump function compactly supported in the interval $(0,4)$.  Given a solution $v$ to \eqref{schrodeqn}, \eqref{bcs} with $f \in L^2(\Omega)$ and a frequency scale $\l \gg 1$, we may define $\beta(-\l^{-2}D_t)v$ as the tempered distribution $\mathscr{F}^{-1}\{\beta(-\l^{-2}\cdot)\mathscr{F}v \}$, where $\mathscr{F}$ denotes the partial Fourier transform in time.  We now let $d(x,\prtl \Omega)$ denote the distance from $x$ to the boundary of $\Omega$ and let $\chi_j$ be a bump function satisfying
\begin{equation}\label{locsmoothchi}
\supp(\chi_j) \subset \{ x: 0 \leq d(x,\prtl \Omega)  \leq 2^{-j+1}\}, \qquad \chi_j \big|_{\{x:d(x,\prtl \Omega) \in [0,2^{-j}]\}} \equiv 1.
\end{equation}
The local smoothing estimates we examine state that if $\l \gg 1$ and $1 \geq 2^{-j}\geq \l^{-\frac 23}$,
\begin{equation}\label{homognrgdk}
\|\beta(-\l^{-2}D_t)\chi_j v \|_{L^2(\RR \times \Omega)} \leq C \l^{-\frac 12} 2^{-\frac j4} \|f\|_{L^2(\Omega)},
\end{equation}
with $C$ independent of $\l$, $j$ and $v$.  Our main theorem states that whenever these estimates are valid, they imply scale invariant Strichartz estimates for subcritical $(p,q)$.

\begin{theorem}\label{thm:stz}
Suppose $\mathcal{K} \subset\RR^n$ is any smooth, compact, strictly convex obstacle and $\Omega = \RR^n\setminus\mathcal{K}$.  If the estimates \eqref{homognrgdk} are satisfied for solutions $v$ to \eqref{schrodeqn}, \eqref{bcs} then the scale invariant Strichartz estimates \eqref{str} are valid provided $(p,q)$ are subcritical, that is
$
\frac 2p + \frac nq < \frac n2.
$
Moreover, if $\frac 2p + \frac nq = \frac n2$, then the estimates \eqref{str} are valid for any $s>0$ (i.e. they hold with an arbitrarily small loss of derivatives).  
\end{theorem}

It is expected that the local smoothing bounds \eqref{homognrgdk} should hold for any domain which is exterior to a strictly convex obstacle.  Indeed, any wave packet at frequency $\l$ concentrated along a glancing ray should spend a time comparable to $\l^{-1}2^{-\frac j2}$ in the support of $\chi_j$.  Taking the square integral in time should thus yield a gain of $\l^{-\frac 12}2^{-\frac j4}$.  It appears to be difficult to prove these estimates in general.  However, as alluded to above, they are valid in the exterior of a ball.
\begin{theorem}[O. Ivanovici]\label{thm:nrgdk}
Suppose $\mathcal{K}$ is the unit ball in $\RR^n$ and $\Omega = \RR^n\setminus\mathcal{K}$.  Let $\Delta$ be the Dirichlet or Neumann Laplacian.  Then given any solution $v$ to \eqref{schrodeqn}, \eqref{bcs} with $f \in L^2(\Omega)$ satisfies the estimates \eqref{homognrgdk}.  Moreover, if $F(s,x) \in L^2(\RR\times \Omega)$ and $u(t,\cdot) = \int_{-\infty}^t e^{-i(t-s)\Delta}\left(\beta(-\l^{-2}D_s) \chi_j F(s,\cdot)\right)\,ds$, we have
\begin{equation}\label{inhomognrgdk}
\|\beta(-\l^{-2}D_t) \chi_j u \|_{L^2(\RR \times \Omega)} \leq C\l^{-1} 2^{-\frac j2} \|\chi_j F\|_{L^2(\RR\times\Omega)}.
\end{equation}
\end{theorem}
This theorem is essentially due to Ivanovici \cite{ivanoballs}, \cite{ivanorevise}.  Since \eqref{homognrgdk}, \eqref{inhomognrgdk}, involve  a slight restatement of the estimates in her work, we make some remarks on the proof of Theorem \ref{thm:nrgdk} in the appendix.

Any Strichartz estimate that results from Theorem \ref{thm:stz} will not be new for Dirichlet boundary conditions.  As mentioned above, they follow from a different result of Ivanovici \cite{ivanomt}, which shows that the full range of estimates are valid.  However, when Neumann boundary conditions are imposed in the exterior of a ball, they expand the range of exponents $(p,q)$ for which the scale invariant estimates are valid when compared to \cite{bssschrod}.  The approach in \cite{ivanomt} uses the Melrose-Taylor parametrix, which yields Strichartz estimates in the Dirichlet case.  However, at the time of this writing, it is unclear that this approach can be effective for Neumann boundary conditions. In the present work, we instead use the parametrix construction in \cite{bssschrod}, which is based on one used for the wave equation by Smith and Sogge in \cite{smithsogge06}.  One of the main steps here is to localize the solution to coordinate charts which flatten the boundary, giving rise to a variable coefficient problem.  The solution and coefficients are then reflected in the boundary, which creates a problem with rough coefficients.  Wave packets can then be used to construct a parametrix for the equation.  In previous works, the virtue of this approach is that it is effective in handling points of convexity and inflection in the boundary of $\Omega$.  This even resulted in sharp $L^p$ estimates on spectral clusters defined on compact domains (see \cite{smithsogge06}).  In the present work, the idea is that since the construction deals with the boundary conditions in a very direct fashion, it can be effective in treating both Dirichlet and Neumann conditions.

\subsection*{Notation} The expression $A \lesssim B$ means that $A \leq C B$ for some implicit constant $C$ depending only on the domain $\Omega$ under consideration and possibly the triple $(p,q,s)$ or indices involved in the inequality.  By the same token, $A \approx B$ means that both $A \lesssim B$ and $B \lesssim A$.  Also, given a Banach space $X$, we will often abbreviate the vector valued $L^p$ space $L^p((-T,T);X)$ by $L^p_T X$.
 
\subsection*{Acknowledgements}
The present work stems from the author's collaborations with Hart Smith and Christopher Sogge.  It is a pleasure to thank them for their insight on boundary value problems and wave packet methods.  The author is also grateful for helpful comments from the anonymous referee.

\section{Strichartz estimates}\label{sec:stz}
In this section, we prove Theorem \ref{thm:stz}.  We focus mainly on the case of Neumann conditions as the adjustments needed for the Dirichlet condition are minor.

\subsection{Preliminary reductions}\label{sec:stzprelim}
Here we reduce the Strichartz estimates of Theorem \ref{thm:stz} to proving inequalities for solutions to a variable coefficient Schr\"odinger equation on $\RR^n$.  The approach here draws from the arguments in \cite[\S2]{smithsogge06}.  In future sections, we will see how wave packets can be used to prove the desired estimates.

It suffices to prove Theorem \ref{thm:stz} under the assumption that $0 < s < \frac 12$.  This is clear when $p,q$ are critical and when $p,q$ are subcritical the full range of desired estimates follows from combining these cases with Sobolev embedding.

Let $\{\phi_j\}_{j=0}^k$ be a smooth partition of unity on $\overline{\Omega}$ such that $d(\supp(\phi_0), \prtl\Omega) >0$ and identically one on a large ball containing $\mathcal{K}$.  Thus when $j\geq 1$, we assume $\phi_j$ is supported in a suitable coordinate chart near the boundary. Since $\phi_0$ vanishes in a neighborhood of $\prtl \Omega$, $\phi_0 v$ solves the following inhomogeneous initial value problem on all of $\RR^n$
$$
(D_t + \Delta)(\phi_0v) =[\Delta,\phi_0]v, \qquad \phi_0 v|_{t=0} = \phi_0 f.
$$
Throughout this work, we will interpret operations such as $[\Delta,\phi_0]$ as the commutator of $\Delta$ with the multiplication operator $v \mapsto \phi_0 v$.  By \cite[Proposition 2.10]{bgtexterior}, Strichartz estimates on $\phi_0u$ follow from the inequalities on $\RR^n$ and the local smoothing estimates on the unit scale~\eqref{locsmooth}.

Fix any $j\geq 1$ and let $\phi=\phi_j$, suppressing $j$ in the notation below.  We may assume that $\phi$ is supported in a neighborhood of $\prtl \Omega$ inside the domain of a boundary normal coordinate chart.  Hence we suppose that $x=(x',x_n)$ forms a coordinate system over $\supp(\phi)$ with $x_n=0$, $x_n >0$ defining the boundary and interior respectively.  In these coordinates, we let $\g_{ij}$ denote the coefficient of the metric tensor formed by pulling back the flat metric.  The boundary normal structure means that $\g_{in} = \delta_{in}$ and hence $x_n = d(x,\prtl \Omega)$.  We denote the Laplace operator acting on a function $h(x)$ in these coordinates as $\Delta_\g$. Using the summation convention and setting $\varrho(x) = \sqrt{\det{\g_{lk}(x)}}$, $D_i = -i\prtl_i$, it takes the form
\begin{equation}\label{coordlap}
\Delta_\g h = \varrho^{-1}(x) D_i \left(\g^{ij}(x)\varrho(x) D_j h \right).
\end{equation}

Taking a sufficiently fine partition of unity above and applying linear transformations if necessary, we may also assume that $\phi$ is supported in $\{x\in \RR^n_+:|x|<1 \}$ and that the domain of the local diffeomorphism defining the coordinates contains $\{x\in \RR^n_+:|x|<3 \}$.   However, we want the $\g^{ij}$ to be defined on all of $\RR^n_+$. To this end, we may assume that $\g^{ij}$ remains unchanged in the set $\{x\in \RR^n_+:|x|\leq 2\}$ but that $\g^{ij}(x)=\delta_{ij}$ for $|x| \geq 3$, as this does not alter the equation for $\phi v$.  Furthermore, we may assume that for some $N$ large and $c_0$ sufficiently small
\begin{equation}\label{c0}
\|\g^{ij}-\delta_{ij}\|_{C^N(\RR^n_+)}
\leq c_0, \qquad  \|\varrho -1\|_{C^N(\RR^n_+)} \leq c_0 .
\end{equation}

Since we may assume that $\phi$ is independent of $x_n$ near the boundary, the function $\phi v(t,\cdot)$ satisfies the Neumann boundary condition $\prtl_n (\phi v)(t,x',0)=0$.  The structure of the boundary normal coordinates allows us to extend $\phi v(t,x)$, $\Delta_\g (\phi v )(t,x)$ and the coefficients $\g^{ij}(x)$, $\varrho(x)$ to all of $\RR^n$ in an even fashion with respect to the boundary hypersurface $x_n=0$ (take an odd extension of $\phi v$ and $\Delta_\g (\phi v)$ for Dirichlet conditions).  Given the boundary condition, the extension defines $\phi v$ as a $C^{1,1}$ function on $\RR^n$ and $\g^{ij}(x)$, $\varrho(x)$ as Lipschitz functions on $\RR^n$.

We next claim that it suffices to show that
\begin{equation}\label{phiv}
\|\phi v\|_{L^p_TL^q(\RR^n)} \lesssim \|\phi v \|_{L^2_T H^{s+\frac 12}(\RR^n)} + \|(D_t + \Delta_\g) (\phi v) \|_{L^2_T H^{s-\frac 12}(\RR^n)},
\end{equation}
where the Sobolev spaces in $\RR^n$ on the right are the usual ones defined using the Fourier transform.  Given \eqref{locsmooth}, this follows by showing that
\begin{equation}\label{omegadom}
\|\phi v \|_{L^2_T H^{s+\frac 12}(\RR^n)} + \|(D_t + \Delta_\g) (\phi v) \|_{L^2_T H^{s-\frac 12}(\RR^n)}
\lesssim \sum_{|\gamma|\leq 2} \|(\prtl^\gamma\phi) v \|_{L^2_T H^{s+\frac 12}(\Omega)}
\end{equation}
which is well defined since we assume that $\phi$ is independent of $x_n$ near the boundary and hence $(\prtl^\gamma \phi) v$ satisfies the boundary condition.

Observe that for any $h$ supported in $\{x: |x|\leq 2\}$ such that $\prtl_n h|_{x_n=0}=0$, $h|_{\RR^n_+}\in C^\infty$, and $h(x',x_n)=h(x',-x_n)$,  we have
$$
\|h\|_{H^2(\RR^n)} \lesssim \|h\|_{L^2(\RR^n)} + \|\Delta_\g h\|_{L^2(\RR^n)} \lesssim \|h\|_{H^2(\RR^n)} .
$$
Indeed, the second inequality here is evident and the first inequality follows from elliptic regularity for operators with Lipschitz coefficients (see e.g. \cite[Theorem 9.11]{GT}).  The middle term is $\approx \|(1+\Delta_\g) h\|_{L^2(\RR^n_+)}$ which gives  $\|h\|_{H^2(\RR^n)} \approx \|h\|_{H^2(\Omega)} $.
We now observe that this implies
$$
\|h\|_{H^r(\RR^n)} \approx \|h\|_{H^r(\Omega)} \qquad \text{for } 0 \leq r \leq 2.
$$
Indeed, interpolating the $H^2$ estimates with the trivial $L^2$ bounds shows this for any $0 \leq r \leq 2$. Already this is enough to bound the first term on the left in \eqref{omegadom}.  Moreover, it is now sufficient to see that
\begin{equation}\label{nablacontrol}
\|(D_t + \Delta_\g) (\phi v) \|_{L^2_T H^{s-\frac 12}(\RR^n)} \lesssim \sum_{|\gamma|\leq 2} \|(\prtl^\gamma\phi) v \|_{L^2_T H^{s+\frac 12}(\RR^n)}
\end{equation}
However, $(D_t + \Delta_\g) (\phi v)=[\Delta_\g,\phi] v$, which can be written as
$$
\varrho^{-1}D_i(\g^{ij}\varrho(D_j\phi)v) + \g^{ij}D_j((D_i\phi) v)-\g^{ij}(D_{ij}\phi)v,
$$
which yields \eqref{nablacontrol} since multiplication by a Lipschitz function preserves $H^{s\pm\frac 12}(\RR^n)$ (given our assumption that $0<s<\frac 12$).

To show \eqref{phiv}, we start with a careful Littlewood-Paley decomposition. Let $\{\beta_l(\zeta)\}_{l=0}^\infty$ be a sequence of smooth functions $\beta_l:[0,\infty) \to [0,1]$ such that
\begin{equation}\label{littpaleyseq}
\sum_{l=0}^\infty \beta_l(\zeta) =1 \text{ for } \zeta \geq 0, \qquad \beta_l(\zeta) = \beta_1(2^{-l+1}\zeta) \text{ for } l\geq 1,
\end{equation}
with $\supp(\beta_0) \subset [0,2)$ and $\supp(\beta_1) \subset (2^{-\frac 12},2^{\frac 32})$.   For $k\geq 0$ define
\begin{align*}
v_{k} &:= \sum_{m =0}^{k+1} \beta_k(|D'|)\beta_m(|D_n|)(\phi v),\\
v_{k,l} &:= \beta_k(|D'|)\beta_l(|D_n|)(\phi v),  \qquad \text{for }0\leq k +2 \leq l < \infty,
\end{align*}
where $\beta_k(|D'|)$, $\beta_l(|D_n|)$ are Fourier multipliers with symbols $\beta_k(|\xi'|)$ and $\beta_l(|\xi_n|)$ respectively. Applying the Littlewood-Paley square function estimate first in $\xi'$, then in $\xi_n$ we have that
\begin{equation}\label{littpaleyest}
\|\phi v\|_{L^p_T L^q}^2 \lesssim \sum_{k=0}^\infty \|v_{k}\|_{L^p_T L^q}^2 +\sum_{k=0}^\infty \sum_{ l=k+2}^\infty \|v_{k,l}\|_{L^p_T L^q}^2
\end{equation}
and since we are reduced to \eqref{phiv} which involves $L^p$ and Sobolev spaces over $\RR^n$, we suppress that dependence here and in what follows. We also pause to observe that since the symbol corresponding to $\beta_l(|D_n|)$ is even and that $\phi v$ is even with respect to $x_n$, we have that
\begin{equation}\label{neumanncoord}
\prtl_n v_k \big|_{x_n =0}=0, \qquad \text{and $v_k$ is even across $x_n=0$}.
\end{equation}
Furthermore, given the compact support of $\phi$ in $\{ |x| \leq 1\}$, we can conclude that
\begin{equation}\label{balldecay}
\left|v_k(t,x)\right| \lesssim
2^{-kN}|x|^{-N}\|(\phi v)(t,\cdot)\|_{L^2}, \qquad |x| \geq 3/2.
\end{equation}

The remainder of this subsection will show that Theorem \ref{thm:stz} is a consequence of a family of estimates on $\|v_{k}\|_{L^p_T L^q}$, $\|v_{k,l}\|_{L^p_T L^q}$.  To state this, let $\g^{ij}_l$, $\varrho_l$ denote the regularized coefficients formed by truncating the $\g^{ij}$ to frequencies less than $c2^l$ for some small $c$. We have the crude estimates (which use that $\g^{ij}$ is Lipschitz)
\begin{equation}\label{firsttruncest}
|\g^{ij}(x)-\g^{ij}_l(x) | \lesssim 2^{-l}, \qquad |\prtl_x^\beta
\g^{ij}_l (x)| \lesssim 2^{l\max(0,|\beta|-1)},
\end{equation}
and similarly for $\varrho_l$.  Also, let $\Delta_{\g_l}$ denote the differential operator formed by replacing the coefficients in \eqref{coordlap} with their regularized counterparts.

\begin{theorem}
Let $p,q,s$ be as in Theorem \ref{thm:stz}.  Then $v_{k}$, $v_{k,l}$ satisfy
\begin{multline}\label{ukredux}
\|v_{k}\|_{L^p_{2^{-k}} L^q} \lesssim \\
2^{ks}\Big(2^{\frac k2}\|v_{k}\|_{L^2_{2^{-k+1}} L^2} + 2^{-\frac k2}\|(D_t + \Delta_{\g_k})v_{k}\|_{L^2_{2^{-k+1}} L^2}+ 2^{-k}\|\phi v\|_{L^2_{2^{-k+1}} L^2}  \Big),
\end{multline}
\begin{equation}\label{uklredux}
\|v_{k,l}\|_{L^p_{2^{-l}} L^q} \lesssim 2^{ls}\left(2^{\frac l2}\|v_{k,l}\|_{L^2_{2^{-l+1}} L^2} + 2^{-\frac l2}\|(D_t + \Delta_{\g_l})v_{k,l}\|_{L^2_{2^{-l+1}} L^2} \right).
\end{equation}
\end{theorem}
To see that these estimates are sufficient, first observe that by time translation and taking a sum over the $\mathcal{O}(2^k)$ intervals in $[-T,T]$ of size $2^{-k+1}$ we may replace the norms $L^p_{2^{-k}} L^q$ and $L^2_{2^{-k+1}} L^2$ in \eqref{ukredux} by $L^p_{T} L^q$ and $L^2_{T} L^2$ respectively.  The same holds for \eqref{uklredux}.  We then use \eqref{littpaleyest} and observe that almost orthogonality and the assumption $s \in [0,\frac 12)$ gives the bound
\begin{multline*}
\sum_{k=0}^\infty \left(2^{2k(s+\frac 12)}\|v_{k}\|_{L^2_T L^2}^2 + 2^{2k(s-1)}\|\phi v\|_{L^2_{T} L^2}^2\right)
 + \sum_{k=0}^\infty \sum_{ l=k+2}^\infty 2^{2l(s+\frac 12)} \|v_{k,l}\|_{L^2_T L^2}^2\\
\lesssim \|\phi v\|_{L^2_T H^{s+\frac 12}}
\end{multline*}
We now turn to the square sum over the $2^{l(s-\frac 12)}\|(D_t + \Delta_{\g_l})v_{k,l}\|_{L^2_{T} L^2}$, which we claim is bounded by the right hand side of \eqref{phiv}.  The analogous one over the $(D_t + \Delta_{\g_k})v_{k}$ is easier and follows similarly.  Let $\beta_{k,l}$ abbreviate the operator $\beta_{k}(|D'|)\beta_l(|D_n|)$
$$
(D_t + \Delta_{\g_l})v_{k,l} = (\Delta_{\g_l}-\Delta_\g)v_{k,l} + [\Delta_\g, \beta_{k,l}]\phi v + \beta_{k,l}(D_t+\Delta_\g)\phi v
$$
To control the term involving $(\Delta_{\g_l}-\Delta_\g)v_{k,l}$, we use \eqref{firsttruncest} and that $\prtl_i\varrho, \prtl_i\g^{ij} \in L^\infty$ to obtain
$$
\|\prtl_i((\varrho_l \g^{ij}_l - \varrho\g^{ij})\prtl_j v_{k,l})\|_{L^2_{T} L^2} \lesssim \|\prtl_j v_{k,l}\|_{L^2_{T} L^2} + 2^{-l}\|\prtl_i\prtl_j v_{k,l}\|_{L^2_{T} L^2} \lesssim 2^l\|v_{k,l}\|_{L^2_{T} L^2} .
$$
Hence almost orthogonality gives that
$$
\sum_{k=0}^\infty \sum_{ l=k+2}^\infty 2^{2l(s-\frac 12)}\left(\|(\Delta_{\g_l}-\Delta_\g)v_{k,l}\|_{L^2_{T} L^2}^2 + \|\beta_{k,l}(D_t+\Delta_\g)\phi v\|_{L^2_T L^2}^2 \right)
$$
is dominated by the right hand side of \eqref{phiv}.

It remains to control the square sum over the $2^{l(s-\frac 12)}\|[\Delta_\g, \beta_{k,l}]\phi v\|_{L^2_T L^2}$, writing
\begin{multline*}
[\Delta_\g, \beta_{k,l}]\phi v = D_i\Big([\g^{ij},\beta_l(|D_n|)]\beta_k(|D'|)D_j(\phi v)\Big)\\ + D_i\Big(\beta_l(|D_n|)[\g^{ij},\beta_k(|D'|)]D_j(\phi v)\Big) + [\varrho^{-1}(D_i\varrho) \g^{ij}, \beta_{l,k}]D_j(\phi v) ,
\end{multline*}
using the summation convention in $i,j$.
This in turn reduces to the 3 bounds
\begin{align}
\sum_{l=k+2}^\infty 2^{2l(s-\frac 12)}\|[\g^{ij},\beta_l(|D_n|)]\,\prtl_j \beta_k(|D'|)(\phi v)\|_{L^2_T H^1}^2 &\lesssim \|\beta_k(|D'|)(\phi v)\|_{L^2_T H^{s+\frac 12}}^2\label{lkhin}
\\
\sum_{k=1}^\infty \|[\g^{ij},\beta_k(|D'|)]\,\prtl_j( \phi v)\|_{L^2_T H^{s+\frac 12}}^2 &\lesssim \|\phi v\|_{L^2_T H^{s-\frac 12}}^2\label{kkhin}
\\
\|\varrho^{-1}(\prtl_i\varrho) \g^{ij}\prtl_j\phi v\|_{L^2_T H^{s-\frac 12}} &\lesssim \|\phi v\|_{L^2_T H^{s+\frac 12}}\label{czbound}
\end{align}
We begin with the last inequality.  Multiplication by the function $\varrho^{-1}(\prtl_i\varrho) \g^{ij}$ is bounded on $H^r(\RR^n)$ for any $0 \leq r < \frac 12$.  This can be seen by the fact that $\prtl_i \varrho$ defines a Calderon-Zygmund type multiplier in $x_n$ and $\langle \xi \rangle^{\frac 12 -\veps}$ defines an $A_2$ weight in one dimension and that the other factors are Lipschitz.  By duality and the assumption that $-\frac 12 < s-\frac 12 < 0$, we have the desired bound \eqref{czbound}.

For \eqref{lkhin} and \eqref{kkhin} we will use Khinchin's inequality and the following fact about the commutator of a Lipschitz function $a$ with a Fourier multiplier $R$ whose symbol lies in $S^z_{1,0}$
\begin{align*}
&[a,R]:H^{z-1}(\RR^m) \to L^2(\RR^m), &  0 \leq z \leq 1,\\
&[a,R]:H^z(\RR^m) \to H^1(\RR^m), &  -1 \leq z \leq 0.
\end{align*}
This was observed in \cite[(2.11), (2.12)]{blairtams} as a consequence of the Coifman-Meyer commutator theorem.  For \eqref{lkhin}, consider an arbitrary sequence $\veps_l = \pm 1$ and let $R=\sum_{l=k+2}^m \veps_{l}2^{l(s-\frac 12)}\beta_{l}(|D_n|)$ where $m > k+2$ is arbitrary.  The operator $R$ is thus a symbol of order $z=s-\frac 12 \in (-\frac 12, 0]$ and hence uniformly in $m$ we have
$$
\|[\g^{ij},R]\,\prtl_j \beta_k(|D'|)(\phi v)\|_{L^2_T H^1}^2 \lesssim \|\prtl_j \beta_k(|D'|)(\phi v)\|_{L^2_T H^{s-\frac 12}}^2.
$$
The bound \eqref{lkhin} now a consequence of Khinchin's inequality.  The remaining bound \eqref{kkhin} follows from similar considerations, this time setting $R=\sum_{k} \veps_{k}\beta_{k}(|D'|)$ (which defines a symbol in $S^0_{1,0}$) and observing that $[\g^{ij},R]:H^{s-\frac 12} \to H^{s+\frac 12}$, a consequence of interpolating the $H^{-1} \to L^2$ and $L^2 \to H^1$ bounds above.

Now that Theorem \ref{thm:stz} is reduced to the bounds \eqref{ukredux}, \eqref{uklredux}, we observe that the latter is a consequence of \cite[Lemma 2.2]{bssschrod}.  Indeed, each $\widehat{v_{k,l}}(t,\cdot)$ is supported in a cone $|\xi'| \leq \frac 32 |\xi_n|$ and hence a semiclassical rescaling $t \mapsto 2^{-l}t$ shows that \eqref{uklredux} shows that follows by taking $\mu = 2^l$ in that lemma (in fact, this bound holds for any $s \geq 0$).

The remainder of this work thus develops the bounds \eqref{ukredux} on $v_k$.  We label $\l = 2^k$ as the frequency scale where $v_k$ is localized. We similarly perform a semiclassical rescaling $t\mapsto  \lambda^{-1} t$ and set $u_\l(t,x) = v_k(\l^{-1} t,x)$, $u(t,x) = \phi v(\l^{-1} t,x)$.  Moreover, we relabel the $\g_k$ as $\g_\l$ so that the Fourier support of the regularized metric is in $\{|\xi| \lesssim \l\}$.  Rescaling \eqref{ukredux} reduces matters to showing that
\begin{equation}\label{uksemi}
\|u_\l\|_{L^p_{1} L^q} \lesssim
\l^{s+\frac 1p} \left(\|u_\l\|_{L^2_{2} L^2} + \|(D_t + \l^{-1}\Delta_{\g_\l})u_\l\|_{L^2_{2} L^2}+ \l^{-\frac {3}2}\|u\|_{L^2_{2} L^2}  \right),
\end{equation}
where we recall that by our convention $L^p_1 L^q$, $L^2_2 L^2$ abbreviate $L^p([-1,1];L^q(\RR^n))$, $L^2([-2,2];L^2(\RR^n))$ respectively.

\subsection{Local smoothing estimates}\label{sec:lsadjust} Here we record a consequence of the local smoothing estimates assumed in the hypothesis of Theorem \ref{thm:stz}.  In what follows, we will take $\chi_j=\chi_j(x_n)$ to be even extensions of the functions in \eqref{locsmoothchi}, that is, $\supp(\chi_j)\subset \{|x_n|\leq 2^{-j+1}\}$ and $\chi_j(x_n) = 1$ whenever $|x_n| \leq 2^{-j}$.  This is because $x_n  = d(x,\prtl \Omega)$ with respect to the metric $\g^{ij}$ when $x_n > 0$ and $|x|\leq 2$.  Moreover, it will be convenient to consider the operator $P_\l$  given by
\begin{equation}\label{Plambdadef}
P_\l u_\l = \l^{-1}\sum_{ij} D_i(\g^{ij}_\l D_j u_\l), \text{ and set } F_\l = (D_t + P_\l)u_\l.
\end{equation}
Since $\l^{-1}\Delta_{\g_\l}-P_\l$ involves only first order derivatives it suffices to prove \eqref{uksemi} with $(D_t + \l^{-1}\Delta_{\g_\l})u_\l$ replaced by $F_\l$.

\begin{proposition}\label{thm:nrgcorollary}
Suppose $u_\l$ is a solution to $(D_t + P_\l)u_\l = F_\l$ satisfying
$$
\supp(\widehat{u_\l}(t,\cdot)) \subset \{\xi: |\xi|\approx \l\},
$$
the boundary condition \eqref{neumanncoord}, and the decay estimate \eqref{balldecay} (with $u_\l$, $u$ replacing $v_k$, $\phi v$).  Then if $2^{-j} \in [\l^{-\frac 23}, 1]$, we have
\begin{equation}\label{locnrgchi}
\|\chi_j u_\l\|_{L^2_1 L^2} \lesssim 2^{-\frac j4}\left(
\|u_\l\|_{L^2_{2} L^2} + \|F_\l\|_{L^2_{2} L^2}+ \l^{-\frac {3}2}\|u\|_{L^2_{2} L^2} \right) .
\end{equation}
\end{proposition}
\begin{proof}
Let $\widetilde{\phi}$ be a smooth cutoff which is identically 1 on $\{x: |x| \leq 3/2 \}$ and supported in $\{x: |x| \leq 2\}$.  Also take $\widetilde{\phi}$ to be even in $x_n$ and independent of $x_n$ near $x_n=0$.  The decay condition \eqref{balldecay} implies that $(1-\widetilde{\phi})u_\l$ satisfies $$
\|(1-\widetilde{\phi})u_\l\|_{L^2_1 L^2 } \lesssim \l^{-\frac 32} \|u\|_{L^2_1 L^2},
$$
meaning it suffices to bound $\widetilde{\phi}u_\l$.  Let $\eta(t)$ be a smooth cutoff identically one on $[-1,1]$, supported in $(-2,2)$ and set  $w=\widetilde{\phi}\eta u_\l$ so that $\|\widetilde{\phi}u_\l\|_{L^2_1 L^2} \lesssim \|w\|_{L^2(\RR^{n+1})}$.  Also let $\widetilde{\beta} \in C_c((0,\infty))$ be a smooth bump function identically one on $\supp(\beta_1)$ (as defined in \eqref{littpaleyseq}).

The function $w$ is even with respect to $x_n$, meaning its values are determined by points $(t,x)$ for which $x_n \geq 0$.  Also, the restriction of $w(t,\cdot)$ to $x_n\geq 0$ can be pulled back to the domain $\Omega$ and in these coordinates, the action of the flat Laplacian on $w(t,\cdot)$ is the same as that of $\Delta_\g$.  Therefore, rescaling the homogeneous local smoothing estimates \eqref{homognrgdk} with $t \mapsto \l^{-1}t$ and Duhamel's principle applied to Cauchy problem with initial time slice $s=-2$, gives
\begin{align*}
\|\widetilde{\beta}(-\l^{-1}D_t) \chi_j w\|_{L^2(\RR^{n+1})} \lesssim 2^{-\frac j4}\|(D_t + \l^{-1}\Delta_\g) w\|_{L^1(\RR; L^2(\RR^n))}.
\end{align*}
Now observe that
\begin{equation*}
(D_t + \l^{-1}\Delta_\g) w = (D_t\eta)\widetilde{\phi}u_\l + \eta[\l^{-1}\Delta_\g,\widetilde{\phi}]u_\l + \widetilde{\phi}\eta(\l^{-1}\Delta_\g-P_\l)u_\l +\widetilde{\phi}\eta F_\l.
\end{equation*}
By the frequency localization of $u_\l$, the estimates \eqref{firsttruncest}, and compact support of the $\eta(t)$, we thus have
$$
\|(D_t + \l^{-1}\Delta_\g) w \|_{L^1(\RR; L^2(\RR^n))} \lesssim \|u_\l \|_{L^2_{2} L^2} + \|F_\l\|_{L^2_{2} L^2}.
$$

It now remains to handle estimates on $(1-\widetilde{\beta}(-\l^{-1}D_t))\eta u_\l$. We use the approach in \cite[Lemma 2.3]{smithsqC2} and \cite[Proposition 2.2]{blairtams} which involves microlocal elliptic regularity.  Let $\tau$, $\xi$ denote Fourier variables dual to $t$, $x$ respectively. Also let $\g^{ij}_{\sqrt{\l}}$ denote the result of truncating the $\g^{ij}$ to frequencies less than $\l^{\frac 12}$, which satisfies the following analog of \eqref{firsttruncest}
\begin{equation}\label{elliptictruncest}
|\g^{ij}_\l(x) - \g^{ij}_{\sqrt{\l}}(x)| \lesssim \l^{-\frac 12}, \qquad |\prtl_x^\beta
\g^{ij}_{\sqrt{\l}} (x)| \lesssim \l^{\frac 12 \max(0,|\beta|-1)}.
\end{equation}
Moreover, let $P_{\sqrt{\l}}$ denote the operator defined by replacing the $\g_\l$ in \eqref{Plambdadef} with the $\g^{ij}_{\sqrt{\l}}$.  Now set
$$
q_\l(x,\tau,\xi) := \frac{(1-\widetilde{\beta}(-\l^{-1}\tau))\widetilde{\beta}(\lambda^{-1}\xi)}{\tau + \lambda^{-1}\sum_{ij}\g^{ij}_{\sqrt{\l}}\xi_i \xi_j},
$$
which is well defined as we may assume that the support of the numerator is disjoint with the zero set of the denominator by taking $c_0$ sufficiently small in \eqref{c0}.  We have that $\l q_\l \in S^0_{1,\frac 12}$ uniformly in $\l$.  The symbolic calculus furnishes a pseudodifferential operator $r_\l(x,D_{t,x})$ with symbol also satisfying $\l r_\l \in S^0_{1,\frac 12}$ such that
$$
(1-\widetilde{\beta}(-\l^{-1}D_t))\eta u_\l = q_\l(x,D_{t,x})(D_t+P_{\sqrt{\l}})\eta u_\l + r_\l(x,D_{t,x})\eta  u_\l.
$$
Using \eqref{elliptictruncest}, $(1-\widetilde{\beta}(-\l^{-1}D_t))\eta u_\l$ satisfies the much stronger estimate
\begin{align*}
\|(1-\widetilde{\beta}(-\l^{-1}D_t))\eta u_\l\|_{L^2(\RR^{n+1})} &\lesssim \l^{-1}\left( \|(D_t+P_{\sqrt{\l}})\eta u_\l\|_{L^2(\RR^{n+1})} +\|\eta u_\l\|_{L^2(\RR^{n+1})} \right)\\
&\lesssim \l^{-\frac 12}\left( \|(D_t+P_\l)u_\l\|_{L^2_2L^2} + \|u_\l\|_{L^2_2L^2} \right)
\end{align*}
where the compact support of $\eta $ is used in the last inequality.
\end{proof}

\subsection{The tangential/nontangential decomposition}
We now begin the discussion of the proof of \eqref{uksemi}.  We are now solely concerned with estimates on $u_\l$, so the notation $v$, $w$, $v_j$ will take on a new meaning for the remainder of the paper.  A crucial step will be a decomposition of the solution $u_\l = v + w$ where the microlocal support of $v(t,\cdot)$ is concentrated in the set \begin{equation}\label{glsupp}
\{\;|\xi_n| \lesssim \l (\l^{-2\alpha} + x_n^2)^{1/4} \},
\end{equation}
and the microlocal support of $w(t,\cdot)$ is concentrated in the set
\begin{equation}\label{ntsupp}
\{\;|\xi_n| \gg \l (\l^{-2\alpha} + x_n^2)^{1/4} \}.
\end{equation}
Here $\alpha < 2/3$ is a parameter which will be chosen below in \eqref{alphachoice}.  The motivation for such a decomposition comes from the bicharacteristics of the equation, that is, the solutions to
$$
\dot{x}_j(t) = 2\g^{ij}\xi_i, \qquad \dot{\xi}_l(t) = -\prtl_l \g^{ij}\xi_i\xi_j.
$$
If a generalized bicharacteristic curve intersects the set \eqref{ntsupp}, it will essentially behave linearly within that set (up to reflections) in the sense that its linear approximation is reasonably accurate.  On the other hand, since the boundary of $\Omega$ is concave, $-\prtl_n \g^{ij}$ defines a positive definite form on vectors $(\xi_1,\dots,\xi_{n-1})$.   Therefore as curves pass through \eqref{glsupp} they will more or less display parabolic behavior.  By this we mean that $x_n(t)$ is convex and that its acceleration is nontrivial.

The function $v$ will be well-suited for a decomposition with respect to distance to the boundary.  Here a key observation is that components of the function which are separated from the boundary will satisfy better Strichartz estimates than those close to the boundary.  This will be counterbalanced by the local smoothing estimates in Proposition \ref{thm:nrgcorollary}, which are arranged so that components of the solution close to the boundary satisfy better smoothing estimates.  On the other hand, $w$ will be well suited for a further decomposition in frequency.  Each component will be microlocalized to a cone of direction vectors, all of which more or less form a common angle to the boundary.  This was a key feature of the approach in \cite{smithsogge06} and subsequently \cite{bssschrod}.  The main idea here is that components whose momentum is concentrated along rays which form a large angle to the boundary satisfy better estimates than components concentrated along rays forming a smaller angle.  However, the losses will once again be counterbalanced by the local smoothing estimates.

For reasons which will be evident later on (see \eqref{lqdispersive}), we define $\sigma(p,q)$
\begin{equation}\label{sigmachoice}
\sigma(p,q):= \begin{cases}
\frac n2 -\frac nq-\frac 2p, & \text{ when  }\;\frac{n-1}{2}(1-\frac 2q) \leq \frac 2p \leq \frac n2 -\frac nq,\\
\frac 12 -\frac 1q, & \text{ when  }\;\frac 2p \leq \frac{n-1}{2}(1-\frac 2q).
\end{cases}
\end{equation}
The precise form of $\sigma(p,q)$ is not all that important in the present work; the crucial feature is that $\sigma(p,q) >0$ if and only if $\frac 2p + \frac nq < \frac n2$.  In this case, we choose $\alpha$ strictly less than, but sufficiently close to, 2/3 so that
\begin{equation}\label{alphachoice}
\frac 1{3\alpha} - \frac{1}{2} < \sigma(p,q) .
\end{equation}
When $\frac 2p + \frac nq = \frac n2$, the difference $\frac 16 - \frac{\alpha}{4}$ will dictate the loss of derivatives in the estimate, hence we take $\alpha<\frac 23$ to make this difference as small as desired.  Taking $\alpha < 2/3$ (rather than $\alpha = 2/3$) will allow us to easily estimate the error which arises by commuting the equation with the microlocal cutoffs to \eqref{glsupp} and \eqref{ntsupp}.  It also ensures that the wave packet parametrix in \S\ref{sec:wpparam} has a bounded error term.  To this end, we pause to observe that taking $\delta=1-\frac{3\alpha}{2}>0$ means that
\begin{equation}\label{calculus}
\l\theta^3 \geq \l^{\delta} \qquad \text{whenever }\;\theta \geq \l^{-\frac{\alpha}{2}}.
\end{equation}

Let $J_\alpha$ be the largest integer such that $2^{-J_\alpha} \geq \l^{-\alpha}$. For $1 \leq j < J_\alpha$, let
\begin{equation}\label{psijdef}
\psi_j(x_n)=\chi_j(x_n)-\chi_{j+1}(x_n), \quad \text{so that} \quad \supp(\psi_j) \subset \{2^{-j}\leq |x_n|\leq 2^{1-j} \}.
\end{equation}
with $\chi_j$ as defined at the beginning of \S\ref{sec:lsadjust}.  Consequently,
$
\chi_0 = \chi_{l} + \sum_{j=0}^{l-1} \psi_l.
$
It will suffice to prove estimates on $\chi_0 u_\l$ since \eqref{balldecay} will yield estimates on $(1-\chi_0 )u_\l$.

Now take a sequence of smooth cutoffs $\{ \Gamma_j\}_{j=1}^{J_\alpha}$ to be applied in the frequency domain such that
$
\sum_{j=1}^{J_\alpha} \Gamma_j(\xi_n)\equiv 1$, with $\supp(\Gamma_{J_\alpha}) \subset \{|\xi_n|\lesssim \l 2^{-\frac{J_\alpha}{2}}\}$ and $\supp(\Gamma_j) \subset \{|\xi_n|\approx \l 2^{-j/2} \}$ for $ 2\leq j < J_\alpha$.  We define $v_{J_\alpha} = \Gamma_{J_\alpha}(\chi_{J_\alpha}u_\l)$ and
$$
w_j = \Gamma_j(D_n)(\chi_ju_\l), \qquad 1 \leq j < J_\alpha,
$$
$$
v_j = \sum_{l=j+1}^{J_\alpha} \Gamma_l(D_n) (\psi_j u_\l), \qquad 1 \leq j < J_\alpha.
$$
Using the properties above, we have that
\begin{align*}
\chi_0 u_\l &= \sum_{l=1}^{J_\alpha} \Gamma_l(D_n)\left( \chi_0 u_\l \right)= \sum_{l=1}^{J_\alpha} \Gamma_l(D_n) \left(\chi_l u_\l + \sum_{j=0}^{l-1}\psi_j u_\l\right)=\sum_{l=1}^{J_\alpha -1} w_l + \sum_{j=1}^{J_\alpha } v_j
\end{align*}
where the last equality follows from a change in the order of summation in $j$ and $l$.

The Fourier support of  $v_j$ is such that
$$
\supp (\widehat{v}_j (t,\cdot) ) \subset \{\;|\xi_n| \lesssim \l2^{-j/2} \},
$$
while the spatial support is concentrated (but not sharply localized) in $\{|x_n| \approx 2^{-j}\}$ as $\l^{-1}2^{\frac j2} \ll 2^{-j}$.  Thus the quantity $2^{-j}$ in some sense dictates the distance to the boundary while $2^{-\frac j2}$ bounds the angle $\xi$ forms with the hyperplane $\xi_n =0$.  On the other hand, the Fourier support of $w_j$ is such that
$$
\supp (\widehat{w}_j (t,\cdot) ) \subset \{\;|\xi_n| \approx \l2^{-j/2} \},
$$
while the spatial support is concentrated in $\{|x_n| \lesssim 2^{-j}\}$.

Recall that we want to prove the following variation on \eqref{uksemi} involving $P_\l$, $F_\l$ as defined in \eqref{Plambdadef},
\begin{equation}\label{Plsemi}
\|u_\l\|_{L^p_{1} L^q} \lesssim
\l^{s+\frac 1p} \left(\|u_\l\|_{L^2_{2} L^2} + \|F_\l\|_{L^2_{2} L^2}+ \l^{-\frac {3}2}\|u\|_{L^2_{2} L^2}  \right).
\end{equation}
We begin by considering the $w_j$.  Let
\begin{equation}\label{ntangdrive}
G_j:=(D_t + P_\l)w_j = [P_\l,\Gamma_j]\chi_ju_\l + \Gamma_j[P_\l,\chi_j] u_\l + \Gamma_j\chi_jF_\l.
\end{equation}
In section \S\ref{sec:wpparam}, we will survey the wave packet parametrix from \cite{bssschrod} and see that it yields the Strichartz estimate
\begin{equation}\label{ntangstz}
\|w_j\|_{L^p_1 L^q} \lesssim \lambda^{s+\frac 1p}2^{-\frac j2 \sigma(p,q)} \left(2^{\frac j4} \|w_j\|_{L^2_1 L^2} + 2^{-\frac j4} \|G_j\|_{L^2_1 L^2} \right).
\end{equation}
with $\frac 2p +\frac nq = \frac n2 -s$. Thus we need the local smoothing estimate
\begin{equation}\label{ntangdk}
2^{\frac j4} \|w_j\|_{L^2_1 L^2 } +
2^{-\frac j4} \|G_j\|_{L^2_1 L^2 }
\lesssim\|u_\l\|_{L^2_{2} L^2 } +
\|F_\l \|_{L^2_{2} L^2 } + \l^{-\frac 32}\|u\|_{L^2_2L^2}.
\end{equation}
If this holds, then we may use that $\sigma(p,q)>0$ when $p,q$ are subcritical to see that $\sum_j \|w_j\|_{L^p L^q}$ is bounded by the right hand side of \eqref{Plsemi}.  Otherwise, when $p,q$ are critical, this sum generates an acceptable logarithmic loss in $\l$.

Proposition~\ref{thm:nrgcorollary} gives that
$$
2^{\frac j4}\|w_j\|_{L^2_1 L^2 } \lesssim 2^{\frac j4}\|\chi_{j}u_\l\|_{L^2_1 L^2 } \lesssim \|u_\l\|_{L^2_{2} L^2 } +
\|F_\l \|_{L^2_{2} L^2 }+ \l^{-\frac 32}\|u\|_{L^2_2L^2}.
$$
and hence it suffices to estimate $G_j$.  The term involving $\Gamma_j\chi_jF_\l$ is trivial to handle.  To estimate the term $[P_\l,\Gamma_j]\chi_j u_\l$ in \eqref{ntangdrive} consider any coefficient $\g^{lm}_\l$ of $P_\l$.  For any function $f(y',y_n)$
\begin{equation}\label{commutatorintegral}
\left([\g^{lm}_\l,\Gamma_j]f\right)(x) = \int \left( \g^{lm}_\l(x',x_n)-\g^{lm}_\l(x',y_n)\right) \check{\Gamma}_j\left(x_n-y_n\right) f(x',y_n)\,dy_n.
\end{equation}
We may assume that $|\check{\Gamma}_j(z)|\lesssim \l 2^{-\frac j2}(1+\l 2^{-\frac j2}|z|)^{-N} $.  Since $\g^{lm}_\l$ is uniformly Lipschitz (see \eqref{firsttruncest}), the mean value theorem and the generalized Young inequality show that this operator on gives rise to a gain of $\l^{-1}2^{\frac j2}$ when acting on $L^2$.  Furthermore, $[P_\l,\Gamma_j]=-\l^{-1}[\g^{lm}_\l,\Gamma_j]\prtl^2_{lm}-\l^{-1}[\prtl_l\g^{lm}_\l,\Gamma_j]\prtl_{m}$ where the sum occurs only over tangential derivatives $\prtl_l$, $l=1,\cdots,n-1$ and the second operator is of lower order.  Hence
$$
2^{-\frac j4}\|[P_\l,\Gamma_j]\chi_j u_\l\|_{L^2_1 L^2 } \lesssim \l^{-2}2^{\frac j4} \|\prtl^2_{lm}(\chi_{j}u_\l)\|_{L^2_1 L^2 } + \|\chi_{j}u_\l\|_{L^2_1 L^2 }  \lesssim 2^{\frac j4} \|\chi_{j}u_\l\|_{L^2_1 L^2 },
$$
and the term on the right hand side can be estimated by Proposition~\ref{thm:nrgcorollary}.  To handle the term $\Gamma_j[P_\l,\chi_j] u_\l$ in \eqref{ntangdrive}, we write the commutator as
$$
[P_\l,\chi_j]u_\l = -2\prtl_n\left((\prtl_n\chi_j) u_\l\right) + (\prtl_n^2\chi_j) u_\l.
$$
The Fourier multiplier $\l^{-1}2^{\frac j2}\Gamma_j(\xi_n)\xi_n$ defines a uniformly bounded operator on $L^2$ and hence
$$
2^{-\frac j4}\|\Gamma_j[P_\l,\chi_j] u_\l\|_{L^2_1 L^2 } \lesssim 2^{\frac j4}\|\chi_{j-1}u_\l\|_{L^2_1 L^2 } + \l^{-1}2^{\frac{7j}{4}} \|\chi_{j-1}u_\l\|_{L^2_1 L^2 }.
$$
Since $\l^{-1}2^{\frac{3j}{2}}\ll 1$ (see \eqref{calculus}), the rest of \eqref{ntangdk} follows from Proposition \ref{thm:nrgcorollary}.

We now consider estimates on the $v_j$. Set $\widetilde{\Gamma}_j = \sum_{j+1}^{J_\alpha}\Gamma_l $ so that
$$
v_j = \widetilde{\Gamma}_j(D_n)(\psi_j u_\l), \qquad v_{J_\alpha} = \Gamma_{J_\alpha}(D_n)(\chi_{J_\alpha} u_\l),
$$
and we may assume $\left|\frac{d^m \widetilde{\Gamma}_j}{d\xi_n}\right|\lesssim (\l2^{-\frac
j2})^{-m}$.  For convenience, define $H_j$ similarly as
\begin{equation*}
H_j:=(D_t + P_\l)v_j = [P_\l,\widetilde{\Gamma}_j]\chi_ju_\l + \widetilde{\Gamma}_j[P_\l,\chi_j] u_\l + \widetilde{\Gamma}_j\chi_jF_\l.
\end{equation*}
In \S\ref{sec:wpparam}, we will see that for $1 \leq j < J_\alpha$
\begin{equation}\label{strslabj}
\|v_j\|_{L^p_1 L^q} \lesssim
\lambda^{s+\frac 1p}2^{-\frac j2 \sigma(p,q)} \Big(2^{\frac j4}\|v_j \|_{L^2_1 L^2} + 2^{-\frac j4} \|H_j\|_{L^2_1 L^2} +\l^{-\frac 32}\|u_\l\|_{L^2_1 L^2}\Big),
\end{equation}
and when $j=J_\alpha$
\begin{equation}\label{strslab0}
\|v_{J_\alpha}\|_{L^p_1 L^q} \lesssim
\lambda^{s+\frac 1p}2^{-\frac{J_\alpha}{2} \sigma(p,q)} \left(\lambda^{\frac 16} \|v_{J_\alpha}\|_{L^2_1 L^2} +
\lambda^{-\frac 16} \|H_{J_\alpha}\|_{L^2_1 L^2} \right),
\end{equation}
with $\frac 2p + \frac nq = \frac n2 -s$ in both cases.

For any coefficient of $P_\l$ we may characterize
the commutator $[\g^{ij}_\l, \widetilde{\Gamma}_j]$ analogously to
\eqref{commutatorintegral}.  As before, the Fourier multiplier
$\l^{-1}2^{-\frac j2}\widetilde{\Gamma}_j(\xi_n)\xi_n$ defines a uniformly bounded operator on $L^2$. Therefore, when $1 \leq j \leq J_\alpha$, the
estimate
\begin{equation}\label{tangdk}
2^{\frac j4}\|v_j \|_{L^2_1 L^2}+ 2^{-\frac j4}\|H_j \|_{L^2_1 L^2} \lesssim \|u_\l\|_{L^2_{2} L^2 } + \|F_\l \|_{L^2_{2} L^2 }+ \l^{-\frac 32}\|u\|_{L^2_2L^2}.
\end{equation}
follows by the essentially the same arguments used to establish \eqref{ntangdk}.  The extra power of decay $2^{-\frac j2 \sigma(p,q)}$ in \eqref{strslabj} allows us to see once again that $\sum_{1}^{J_\alpha-1} \|v_j\|_{L^p L^q}$ is bounded by the right hand side of \eqref{Plsemi} when $\sigma(p,q)>0$.  Otherwise this sum generates a logarithmic loss.  The loss of $\l^{\frac 16}$ in \eqref{strslab0} is larger than the gain of $2^{-\frac{J_\alpha}{4}} \approx \l^{-\frac{\alpha}{4}}$ given by \eqref{tangdk}.  However, by the choice of $\alpha$ in \eqref{alphachoice},
\begin{equation*}
\l^{\frac 16}2^{-\frac{J_\alpha}{4} -\frac{J_\alpha}{2}\sigma(p,q)}\approx \lambda^{\frac 16 - \frac{\alpha}{4} - \frac{\alpha}{2}\sigma(p,q)} \leq
\begin{cases}
1 & \text{when  }\frac 2p + \frac nq < \frac n2\\
\lambda^{\frac 16 - \frac{\alpha}{4}}  & \text{when  }\frac 2p + \frac nq = \frac n2
\end{cases}.
\end{equation*}
Therefore, the estimate \eqref{tangdk} allows us to conclude
$$
\|v_{J_\alpha}\|_{L^p_1 L^q} \lesssim \l^{s+\frac 1p}\left(\|u_\l\|_{L^2_{2} L^2 } + \|F_\l \|_{L^2_{2} L^2 } + + \l^{-\frac 32}\|u\|_{L^2_2L^2}\right)
$$
when $\frac 2p + \frac nq < \frac n2$ and the same bound with a loss of $\frac 16 -\frac{\alpha}{4}$ derivatives when $\frac 2p + \frac nq = \frac n2$.

\subsection{Proving the Strichartz estimates}\label{sec:wpparam}
In this section, we discuss the proofs of \eqref{ntangstz}, \eqref{strslabj}, and \eqref{strslab0}.  We first sketch the proof of the bounds for $v_{J_\alpha}$ in \eqref{strslab0}, which are a mild adjustment of arguments in \cite[\S4]{bssschrod}.  The other estimates will follow by similar considerations.  Let $P_{\l^{2/3}}$ be the operator obtained by regularizing the coefficients of $P_\l$, truncating them to frequencies less than $\l^{\frac 23}$.  Given coefficients $\g_{\l}$, $\g_{\l^{2/3}}$ of $P_\l$, $P_{\l^{2/3}}$ respectively, we have
\begin{equation}\label{lambda23}
|\g_{\l}(x) - \g_{\l^{2/3}}(x)| \lesssim \l^{-\frac 23} \qquad \text{and} \qquad |\prtl^\beta_x \g_{\l^{2/3}}(x)| \lesssim \l^{\frac 23 \max(0,|\beta|-1)}.
\end{equation}
This means that it is sufficient to prove \eqref{strslab0} with $P_\l$ replaced by $P_{\l^{2/3}}$ as the error can be absorbed in to the term $\l^{\frac 16}\|v_{J_\alpha}\|_{L^2_\veps L^2}$.

We now dilate the problem in space-time by $(t,x) \mapsto (\l^{-\frac 13}t, \l^{-\frac 13}x)$.  Let
$$
v(t,x) = v_{J_\alpha}(\l^{-\frac 13}t, \l^{-\frac 13}x), \qquad F(t,x) = \l^{-\frac 13}\left((D_t + P_{\l^{2/3}}) v_{J_\alpha}\right)(\l^{-\frac 13}t, \l^{-\frac 13}x).
$$
Setting $\mu = \l^{\frac 23}$, it thus suffices to show the rescaled estimate
\begin{equation*}
\|v\|_{L^p_\veps L^q} \lesssim \mu^{s+\frac 1p}2^{-\frac{J_\alpha}{2} \sigma(p,q)} \left(\|v\|_{L^2_{\veps} L^2} + \|F\|_{L^2_{\veps}L^2}\right).
\end{equation*}
Indeed, if this estimate holds over the slab $[-\veps,\veps] \times \RR^n$, then it will also hold for translated slabs $[(k-1)\veps,(k+1)\veps] \times \RR^n$.  This yields the estimate over $[-\l^{\frac 13},\l^{\frac 13}]\times \RR^n$ (and subsequently \eqref{strslab0}) after taking a sum in $k$.

Since $P_{\l^{2/3}}$ is self adjoint, we have that differentiating $\|v(t,\cdot)\|^2_{L^2}$ in $t$ gives
\begin{equation}\label{adjoint}
\|v\|_{L^\infty_{\veps} L^2} \lesssim \|v\|_{L^2_{\veps} L^2} + \|F\|_{L^1_{\veps}L^2},
\end{equation}
Furthermore, $\|F\|_{L^1_{\veps}L^2}\lesssim  \|F\|_{L^2_{\veps}L^2}$. It thus suffices to show that
\begin{equation}\label{rescaleJ}
\|v\|_{L^p_\veps L^q} \lesssim
\mu^{s+\frac 1p}\theta_0^{\sigma(p,q)} \left(\|v\|_{L^\infty_{\veps} L^2} + \|F\|_{L^1_{\veps}L^2}\right), \qquad \theta_0 := 2^{-\frac{J_\alpha}{2}},
\end{equation}
where the $\theta$ notation is used here and below to align this work with \cite{smithsogge06} and \cite{bssschrod}.

The inequality \eqref{rescaleJ} follows from wave packet methods. To
this end, we let $g$ be a fixed, real valued, radial Schwartz class function
with $\widehat{g}$ compactly supported in a small ball. Furthermore, we take $g$ to be normalized so that $\|g\|_{L^2} =
(2\pi)^{-\frac n2}$.  With this, we define the operator $T_\mu$ on
Schwartz class functions by
\begin{equation}\label{tmudef}
(T_\mu f)(x,\xi) = \mu^{\frac n4} \int e^{-i\langle \xi, y-x \rangle} g(\mu^{\frac 12}(y-x)) f(y)\,dy.
\end{equation}
The normalization ensures that $T^*_\mu T_\mu =I$ and $\|T_\mu f\|_{L^2(\RR^{2n}_{x,\xi})} = \|f\|_{L^2(\RR^{n}_{y})}$.  Let
$$
\tilde{v}(t,x,\xi) = (T_\mu v(t,\cdot))(x,\xi) .
$$
Recall that since $\supp(\widehat{v}_{J_\alpha}(t,\cdot))\subset \{|\xi_n| \lesssim  \l 2^{-\frac{J_\alpha}{2}} \}$, we have that the rescaled function satisfies $\supp(\widehat{v}(t,\cdot))\subset \{|\xi_n| \lesssim \mu \theta_0 \}$.  The compact support of $g$ allows us to assume
\begin{equation}\label{wpsupportv}
\supp(\tilde{v}(t,x,\cdot)) \subset \{|\xi|\approx \mu, |\xi_n|
\lesssim  \mu\theta_0 \}.
\end{equation}

In this subsection, let $q$ be the symbol defined by $q(x,\xi) = \l^{-\frac 13} p_{\l^{2/3}}(\l^{-\frac 13} x, \l^{\frac 13} \xi)$ where $p_{\l^{2/3}}$ is the symbol of $P_{\l^{2/3}}$. Hence $(D_t+Q(x,D))v =F$ and
$$
|\prtl^\beta_x \prtl^\gamma_\xi q(x,\xi)| \lesssim_{\alpha, \beta} \mu^{1-|\gamma|+\frac 12\max(0,|\beta|-2)} \qquad |\xi|\approx \mu.
$$
It is shown in \cite[(29)]{bssschrod} (and similarly in
\cite[(3.1)]{bsspams}) that we may write
\begin{equation}\label{conjugeqn}
\left(\prtl_t - d_\xi q(x,\xi) \cdot d_x + d_x q(x,\xi)\cdot
d_\xi+iq(x,\xi) -i\xi\cdot d_\xi q(x,\xi) \right) \tilde{v}(t,x,\xi)
= \tilde{F}(t,x,\xi),
\end{equation}
where $\tilde{F}$ is supported in the same set appearing in \eqref{wpsupportv} and satisfies
$$
\|\tilde{F}\|_{L^1_\veps L^2(\RR^{2n}_{x,\xi})} \lesssim \|v\|_{L^\infty_{\veps} L^2(\RR^n)} +
\|F\|_{L^1_{\veps}L^2(\RR^n)}.
$$

Now let $\Theta_{r,t}(x,\xi) = (x_{r,t}(x,\xi),\xi_{r,t}(x,\xi))$ be the time $r$ solution of initial value problem for Hamilton's equations
$$
\dot{x} = d_\xi q(x,\xi), \qquad \dot{\xi} = -d_xq(x,\xi), \qquad (x(t),\xi(t)) = (x,\xi).
$$
Observe that since $q$ is independent of time, $\Theta_{r,t}(x,\xi)=\Theta_{0,t-r}(x,\xi)$.  Define
$$
\psi(t,x,\xi) = \int_0^t \big[ q(\Theta_{s,t}(x,\xi)) -\xi_{r,t}(x,\xi)\cdot d_\xi q(\Theta_{s,t}(x,\xi))\big] \,ds.
$$
This allows us to write
$$
\tilde{v}(t,x,\xi) =
e^{-i\psi(t,x,\xi)}\tilde{v}(0,\Theta_{0,t}(x,\xi)) + \int_0^t
e^{-i\psi(t-r,x,\xi)} \tilde{F}(r,\Theta_{0,t-r}(x,\xi))\,dr.
$$

We now define an operator $W$ acting on functions $\tilde{f} \in L^2(\RR^{2n}_{x,\xi})$
satisfying a support condition in $\xi$ of the form
\begin{equation}\label{wpsupportf}
\supp(\tilde{f}(x,\cdot))\subset\{|\xi| \approx \mu, |\xi_n|\lesssim
\mu\theta_0 \}.
\end{equation}
For such functions we take
\begin{equation}\label{Wdef}
W\tilde{f}(t,x) = T^*_\mu \left[ \tilde{f}\circ\Theta_{0,t}\right](x).
\end{equation}
Since $T_\mu$ is an isometry and
$(x,\xi) \mapsto \Theta_{0,t}(x,\xi)$ is a
measure preserving diffeomorphism, it now suffices to show that
\begin{equation}\label{Wmap}
\|W\tilde{f}\|_{L^p_\veps L^q(\RR^n)} \lesssim \mu^{s+\frac 1p}
\theta_0^{\sigma(p,q)}\|\tilde{f}\|_{L^2(\RR^{2n})}.
\end{equation}
By duality, this is equivalent to
\begin{equation}\label{Wdual}
\|WW^*G\|_{L^p_\veps L^q(\RR^n)} \lesssim \mu^{2(s+\frac 1p)}
\theta_0^{2\sigma(p,q)}\|G\|_{L^{p'}_\veps L^{q'}(\RR^n)}.
\end{equation}
Let $W_t$ be the fixed time operator $W_t \tilde{f}=W\tilde{f}(r,x)\big|_{r=t}$.
Similar to \cite[(4.2), (4.3)]{bssschrod}, we have the pair of
estimates
\begin{equation}\label{Wdispersive}
\|W_rW^*_t\|_{L^1 \to L^\infty} \lesssim \mu^{\frac
n2}(\mu^{-1}+|t-r|)^{-\frac{n-1}{2}}(\mu^{-1}\theta_0^{-2}+|t-r|)^{-\frac{1}{2}},
\end{equation}
\begin{equation}\label{Wenergy}
\|W_rW^*_t\|_{L^2 \to L^2} \lesssim 1.
\end{equation}
Indeed, if such estimates hold, then interpolation gives
\begin{equation}\label{Winterpolate}
\|W_rW^*_t\|_{L^{q'} \to L^{q}} \lesssim \mu^{\frac n2(1-\frac 2q)} (\mu^{-1}+|t-r|)^{-\frac{n-1}{2}(1-\frac 2q)}(\mu^{-1}\theta_0^{-2}+|t-r|)^{-\frac{1}{2}(1-\frac 2q)}.
\end{equation}
When $\frac 2p + \frac nq = \frac n2$ we use the full strength of the time decay to obtain
$$
\|W_rW^*_t\|_{L^{q'} \to L^{q}} \lesssim \mu^{\frac 2p}|t-r|^{-\frac 2p}.
$$
The bound \eqref{Wdual} then follows from the Hardy-Littlewood-Sobolev theorem of fractional integration.  When $\frac 2p + \frac nq < \frac n2$, we sacrifice as much time decay as possible in the last factor on the right of \eqref{Winterpolate} to obtain
\begin{equation}\label{lqdispersive}
\|W_rW^*_t\|_{L^{q'} \to L^{q}} \lesssim \mu^{2(s+\frac 1p)}\theta_0^{2\sigma(p,q)}|t-r|^{-\frac 2p}.
\end{equation}
Thus \eqref{Wdual} follows again by fractional integration.

The estimate \eqref{Wenergy} is a consequence of the fact that $T_\mu$ is an isometry and $(x,\xi) \mapsto \Theta_{0,t}(x,\xi)$ is measure preserving.  Hence we turn our attention to \eqref{Wdispersive} and outline its proof.  As in \cite[\S4]{bssschrod}, the action of $W_r W_t^*$ on a function $G(t,y)$ can be characterized as integration against a kernel $K(r,x;t,y)$ that takes the form
\begin{equation}\label{packetkernel}
\mu^{\frac n2} \int e^{i\langle\zeta, x-z\rangle - i
\psi(r-t,x,\zeta) -i\langle \zeta_{t,r, y-z_{t,r}}\rangle}
g(\mu^{\frac 12}(y-z_{t,r}))g(\mu^{\frac
12}(x-z))\Upsilon(\zeta)\,dz \,d\zeta,
\end{equation}
where $\Upsilon(\zeta)$ is a harmless smooth cutoff to the region in \eqref{wpsupportf}.  The desired estimate \eqref{Wdispersive} thus follows from the bound
\begin{equation}\label{kernelbound}
|K(r,x;t,y)| \lesssim \mu^{\frac n2}
(\mu^{-1}+|t-r|)^{-\frac{n-1}{2}}(\mu^{-1}\theta_0^{-2}+|t-r|)^{-\frac{1}{2}}.
\end{equation}

The proof of estimate \eqref{kernelbound} follows by the same
methods as in \cite[\S4]{bssschrod}.  The only difference is that
the angular parameter $\theta_0$ is larger than what is used in that
work (there it is assumed that $\theta_0=\mu^{-\frac 12}$). However, the proof is easily modified to handle this
situation.  We motivate the main idea here using the principle of stationary phase.  However, since the derivatives of phase and amplitudes involved depend on $\mu$, the precise arguments from \cite[\S4]{bssschrod} are required.

As observed in \cite[\S4]{bssschrod}, \cite[p.254]{bsspams}, $-\prtl_{\zeta_i}\psi(r-t,x,\zeta)  + \zeta_{t,r}\cdot\prtl_{\zeta_i}z_{t,r}=0$ for any $1 \leq i \leq n$. Therefore applying $-id_\zeta$ to the phase in \eqref{packetkernel} gives
$$
x-z-d_\zeta \zeta_{t,r}\cdot( y-z_{t,r}(z,\zeta)).
$$
Since $g(\mu^{\frac 12}(x-z))$ is highly concentrated near $x=z$,
this differential can be well approximated by
$$
-d_\zeta \zeta_{t,r}\cdot (y-z_{t,r}(x,\zeta)),
$$
which has a critical point when $y=z_{t,r}(x,\zeta)$ for some $\zeta$. The Hessian is now approximately
\begin{equation}\label{phasehessian}
-d_\zeta^2 \zeta_{t,r}\cdot (y-z_{t,r}(x,\zeta))+d_\zeta
\zeta_{t,r}\cdot d_\zeta z_{t,r}(x,\zeta).
\end{equation}
It can then be reasoned that $z_{t,r}(x,\zeta) \approx z+2\mu^{-1}(t-r)\zeta$ and that $\zeta_{t,r}(z,\zeta) \approx \zeta$ and this approximation behaves well under differentiation.  Hence the first term in \eqref{phasehessian} is small relative to the second, which is essentially $-2\mu^{-1}(t-r)I$.  This illustrates why the critical point is nondegenerate.

To see \eqref{kernelbound}, treat the cases $0\leq |t-r| \leq \mu^{-1}$, $\mu^{-1}\theta_0^{-2} \leq |t-r| \leq \veps$, and $\mu^{-1} \leq |t-r| \leq \mu^{-1}\theta_0^{-2}$ separately.  In the first case, $0\leq |t-r| \leq \mu^{-1}$, we see that
$$
|K(r,x;t,y)| \lesssim \mu^{n}\theta_0,
$$
which is sufficient and does not use the oscillations of the phase. It just uses that the integral in $z$ is uniformly bounded and the integral over $\zeta$ gives the volume of the set in \eqref{wpsupportf}, which is $\approx \mu^n\theta_0$.  When $\mu^{-1}\theta_0^{-2} \leq |t-r| \leq \veps$, it suffices to show
$$
|K(r,x;t,y)| \lesssim \mu^{\frac n2}|t-r|^{-\frac n2}.
$$
Since the Hessian behaves like $-2\mu^{-1}(t-r)I$, this is a typical application of the principle of stationary phase.  In the final case $\mu^{-1} \leq |t-r| \leq \mu^{-1}\theta_0^{-2}$, we obtain better estimates by applying stationary phase in the variables $\zeta_1,\dots,\zeta_{n-1}$ and ignoring any oscillations in $\zeta_n$. The structure of the Hessian is amenable to such an approach.  Since the $\zeta_n$ support of the set \eqref{wpsupportf} has volume $\approx \mu \theta_0$, we obtain
$$
|K(r,x;t,y)| \lesssim
\left(\mu\theta_0\right)\mu^{\frac{n-1}2}|t-r|^{-\frac{n-1}2}.
$$

We now turn to the estimates \eqref{strslabj} when $1 \leq j < J_\alpha$, and set $\theta = 2^{-\frac j2}$, again to align notation with prior works.  Recall that $v_j =\widetilde{\Gamma}_j(D_n)\psi_j u_\l$ and hence
\begin{equation*}
v_j(t,x) = \l 2^{-\frac j2}\int \check{\widetilde{\Gamma}}_2\left(\l2^{-\frac j2}(x_n-y_n)\right) \psi_j(y_n)u_\l(t,x',y_n)\,dy_n.
\end{equation*}
Given that $\psi_j (y_n)$ is supported where $|y_n| \approx 2^{-j}=\theta^2$, and the convolution kernel is rapidly decaying on the much smaller scale $\l^{-1}2^{\frac j2} = \l^{-1}\theta^{-1}$, it can be seen from this that for $M$ sufficiently large
\begin{equation}\label{boundarysep}
\| \langle \l^{\frac 12} \theta^{-\frac 12} x_n \rangle^{-M}v_j\|_{L^2_{\veps}L^2} \lesssim (\l \theta^3)^{-M}\| u_\l\|_{L^2_{\veps}L^2} \lesssim \l^{-M\delta}\|u_\l\|_{L^2_{\veps}L^2}.
\end{equation}
Let $\g^{kl}_j$ be the coefficients obtained by truncating the $\g_\l$ to frequencies $c \l^{\frac 12}\theta^{-\frac 12}$ for some small constant $c$.  We observe that
\begin{align}
|\prtl^\beta_{x'} \prtl^m_{x_n}
(\g^{kl}_\l-\g^{kl}_j)(x)|\lesssim (\l^{\frac 12}\theta^{-\frac 12})^{m-1}\langle\l^{\frac 12}\theta^{-\frac 12} x_n
\rangle^{-M}, & & m=0,1\label{truncestapprox}\\
|\prtl^\beta_{x'} \prtl^m_{x_n} \g^{kl}_j(x)|  \lesssim
c_0\big(1+(\l^{\frac 12}\theta^{-\frac 12})^{\max(0,m-1)}\langle \l^{\frac 12}\theta^{-\frac 12} x_n \rangle^{-M}\big), & &|\beta| + m \geq 1. \label{truncestderiv}
\end{align}
The proof of the latter pair of estimates follow by the same considerations as in \cite[(6.31), (6.32)]{smithsogge06} (taking $\mu^{\frac 12} = \l^{\frac 12}\theta^{-\frac 12}$ there).  The main idea here is that the most singular part of $\prtl_n^2 (\g^{ij}(x',|x_n|))$ behaves like $\prtl_n \g^{kl}(x',0) \cdot\delta(x_n)$.  Hence $\g^{kl}_j$ is essentially smooth outside a $\l^{-\frac 12}\theta^{\frac 12}$ neighborhood of $x_n=0$.

Now let $P_{j}$ denote the differential operator obtained by of $P_\l$ with the $\g^{kl}_j$.  By \eqref{boundarysep}, \eqref{truncestapprox} we have
\begin{equation}\label{Pjapprox}
\| (P_j-P_\l)v_j\|_{L^1_{\veps}L^2}
\lesssim \l^{\frac 12} \theta^{\frac 12}(\l \theta^3)^{-M}\|
u_\l\|_{L^2_{\veps}L^2} \lesssim \l^{\frac 12-M\delta}\|u_\l\|_{L^2_{\veps}L^2}.
\end{equation}
Let $F_j(t,x) = (D_t + P_j)v_j(t,x)$, and observe that it now suffices to show that
\begin{equation}\label{Pjineq}
\|v_j\|_{L^p_\veps L^q}
\lesssim \l^{s+\frac 1p} \theta^{\sigma(p,q)}\Big( \theta^{-\frac 12}\|v_j\|_{L^2_\veps L^2} + \l^{\frac 12}\theta\|\langle \l^{\frac 12}\theta^{-\frac 12} x_n \rangle^{-M}v_j\|_{L^2_\veps L^2} + \theta^{\frac 12}\|F_{j}\|_{L^2_\veps L^2} \Big).
\end{equation}
Indeed, by taking $M$ sufficiently large in the estimate \eqref{Pjapprox}, the error $(P_\l -P_j)v_j$ can be absorbed into the term $\l^{-\frac 32}\|u_\l\|_{L^2_{\veps}L^2}$ in \eqref{strslabj} and it suffices to consider the equation involving $P_j$.  Similarly, we can take $M$ large in \eqref{Pjineq} so that \eqref{boundarysep} will ensure that the term involving $\langle \l^{\frac 12}\theta^{-\frac 12} x_n \rangle^{-M}v_j$ can also be bounded by $\l^{-\frac 32}\|u_\l\|_{L^2_{\veps}L^2}$.

We now rescale the space time variables by $(t,x) \mapsto (\theta t,
\theta x)$ and set $\mu = \l\theta$.  Let
$$
v(t,x) := v_j(\theta t,\theta x),
\qquad F(t,x) := \theta \left((D_t + P_j)v_j\right)(\theta t, \theta x).
$$
Since we are working with a fixed index $j$, we suppress the
dependence on $\mu$, $j$ in these definitions.  Furthermore, let
$Q(x,D)$ be defined by the symbol $q(x,\xi) = \theta p_j(\theta x,
\theta^{-1}\xi)$ where $p_j$ is the symbol of $P_j$. We now have a
solution to
$$
\left(D_t + Q \right)v = F,
$$
and symbol of $Q$ satisfies (cp. \cite[(26)]{bssschrod} and \eqref{truncestderiv} above)
\begin{equation}\label{qsymbolest}
\left| \prtl^\beta_x \prtl^\alpha_\xi q(x,\xi)\right| \lesssim
\begin{cases} \mu^{1-|\alpha|}, & \text{if } |\beta|
=0,\\
c_0\left(1+\mu^{(|\beta|-1)/2}\theta\langle \mu^{\frac 12}x_n
\rangle^{-N}\right) \mu^{1-|\alpha|}, & \text{if }|\beta|\geq 1.
\end{cases}
\end{equation}

Rescaling the estimate \eqref{Pjineq}, it suffices to show
\begin{equation*}
\|v\|_{L^p_\veps L^q} \lesssim \mu^{s+\frac 1p}
\theta^{\sigma(p,q)}\left( \|v\|_{L^2_\veps L^2} + \mu^{\frac
12}\theta\|\langle \mu^{\frac 12} x_n \rangle^{-M}v\|_{L^2_\veps
L^2} +\|F\|_{L^2_\veps L^2}
\right),
\end{equation*}
because we can reason as before to see this yields an estimate on $[-\theta^{-1}\veps, \theta^{-1}\veps] \times \RR^n$.  We can again argue as in \eqref{adjoint}, to see that it suffices to show
\begin{equation}\label{rescalej}
\|v\|_{L^p_\veps L^q} \lesssim \mu^{s+\frac 1p}
\theta^{\sigma(p,q)}\left( \|v\|_{L^\infty_\veps L^2} + \mu^{\frac
12}\theta\|\langle \mu^{\frac 12} x_n \rangle^{-M}v\|_{L^2_\veps
L^2} + \|F\|_{L^1_\veps L^2}
\right).
\end{equation}
This estimate again follows using wave packet methods, however here we must take additional care as $\mu^{-1}q(x,\xi)$ is not uniformly $C^2$ in $x$.  Instead we use Lemma 4.3 from \cite{smithsogge06} which shows how to conjugate the operator $Q$ by the wave packet transform $T_\mu$.  This lemma shows that
\begin{multline}\label{schwartzmod}
\big(q(y,D_y)^* - id_\xi(x,\xi)\cdot d_x + id_xq(x,\xi)\cdot d_\xi\big) \left[ e^{i\langle \xi, y-x\rangle} g(\mu^{\frac 12}(y-x))\right]\\
=e^{i\langle \xi, y-x\rangle} g_{x,\xi}(\mu^{\frac 12}(y-x)),
\end{multline}
where $g_{x,\xi}(\cdot)$ is a family of Schwartz class functions depending on $(x,\xi)$ and with $\widehat{g}_{x,\xi}$ also supported in a small ball. In addition, if $\|\cdot\|$ is any Schwartz seminorm, we have the estimate
\begin{equation}\label{seminorm}
\|g_{x,\xi}\| \lesssim 1+c_0\mu^{\frac 12}\theta\langle \mu^{\frac 12} x_n \rangle^{-2M}.
\end{equation}
Strictly speaking this lemma is stated for $M=3$, but the rapid
decay of the symbol estimates in \eqref{qsymbolest} means that the
same proof works for any $M>0$.  Analogous to \eqref{conjugeqn}, we
take $\tilde{v}(t,x,\xi)$ as the wave packet transform of
$v(t,\cdot)$, but this time let $\tilde{F}(t,x,\xi)$ denote the transform of $F(t,\cdot)$.  The function $\tilde{v}(t,x,\xi)$
satisfies
\begin{multline*}
\left(\prtl_t - d_\xi q(x,\xi) \cdot d_x + d_x q(x,\xi)\cdot
d_\xi+iq(x,\xi) -i\xi\cdot d_\xi q(x,\xi) \right) \tilde{v}(t,x,\xi)
=\\ \tilde{F}(t,x,\xi)+\tilde{G}(t,x,\xi),
\end{multline*}
where
$$
\tilde{G}(t,x,\xi)=\mu^{\frac n4} \int e^{-i\langle \xi, y-x
\rangle} g_{x,\xi}(\mu^{\frac 12}(y-x)) v(t,y)\,dy,
$$
and $g_{x,\xi}$ is the family of Schwartz functions in
\eqref{schwartzmod}.  Again by the compact support of $\widehat{g}$ and $\widehat{g}_{x,\xi}$, we may assume that
\begin{equation*}
\supp(\tilde{v}(t,x,\cdot)), \;\supp(\tilde{F}(t,x,\cdot)), \; \supp(\tilde{G}(t,x,\cdot)) \subset \{\xi: |\xi|\approx \mu, |\xi_n| \lesssim  \mu\theta\}.
\end{equation*}

We further decompose
$\tilde{G}=\tilde{G}_1+\tilde{G}_2$ where $\tilde{G}_1=\eta
\tilde{G}$ and $\eta=\eta(x_n)$ is supported where $\mu^{\frac
12}\theta\langle \mu^{\frac 12}x_n \rangle^{-M} \geq \frac 12$ and
$\eta\equiv 1$ on the set where $\mu^{\frac 12}\theta\langle
\mu^{\frac 12}x_n \rangle^{-M} \geq 1$.  Therefore
$$
\tilde{G}_1(t,x,\xi)=\eta(x_n)\mu^{\frac n4} \int e^{-i\langle \xi,
y-x \rangle} g_{x,\xi}(\mu^{\frac 12}(y-x)) v(t,y)\,dy.
$$
We claim that
\begin{align}
\|\tilde{G}_1\|_{L^1_\veps L^2(\RR^{2n}_{x,\xi})} &\lesssim \mu^{\frac
12}\theta\|\langle \mu^{\frac 12} x_n \rangle^{-M}v\|_{L^2_\veps
L^2(\RR^n)},\label{G1est}\\
\|\tilde{G}_2\|_{L^1_\veps L^2(\RR^{2n}_{x,\xi})} &\lesssim \|v\|_{L^2_\veps L^2(\RR^n)}
.\label{G2est}
\end{align}
To see \eqref{G1est}, we observe that by duality ($TT^*$) it suffices to show
that
$$
\left\|\int \widetilde{K}(y,\eta;x,\xi) h(x,\xi)\,dxd\xi\right\|_{L^2(\RR^{2n}_{y,\eta})}
\lesssim \mu \theta^2\left\| h \right\|_{L^2(\RR^n_{x,\xi})},
$$
where $\widetilde{K}(y,\eta;x,\xi) $ is defined as
$$
\mu^{\frac n2} e^{i\langle \eta, y\rangle -i\langle \xi, x \rangle
}\eta(x_n)\eta(y_n)\int e^{i\langle \xi-\eta,z\rangle}
g_{y,\eta}(\mu^{\frac 12}(z-y))g_{x,\xi}(\mu^{\frac 12}(z-x))\langle
\mu^{\frac 12}z_n\rangle^{2M}\,dz
$$
(cp. \cite[Lemma 4.2]{smithsogge06}). We may now integrate by parts and use \eqref{seminorm} to obtain
$$
\left|\widetilde{K}(y,\eta;x,\xi)\right| \lesssim \mu\theta^2\left( 1+\mu^{-\frac
12}|\eta-\xi| + \mu^{\frac 12}|x-y| \right)^{-(2n+1)},
$$
and the desired estimate follows.  The bound \eqref{G2est} follows
similarly.  This shows that the parametrix has bounded error relative to the spaces on the right hand side of \eqref{rescalej}.

The estimate \eqref{rescalej} now follows similarly to the one in \eqref{rescaleJ}.  Indeed, it now suffices to define the map $W$ as in \eqref{Wdef} and prove \eqref{Wdual} with $\theta=2^{-\frac j2}$ replacing $\theta_0$.  The latter estimate is now a consequence of \eqref{kernelbound} with the same replacement, which in turn follows from the same considerations as before. Indeed, since we can take the $\zeta$ integral in \eqref{packetkernel} to be supported in a region of the form $\{ |\zeta| \approx \mu, |\zeta_n| \lesssim \mu\theta\}$, the desired estimates are a consequence of the arguments in \cite{bssschrod}.

The estimates \eqref{ntangstz} follow more directly from the results in \cite{bssschrod}.  Indeed, since the frequency support of the $w_j$ are localized to a set where $|\xi_n| \approx \l 2^{-\frac j2}$, this is a consequence of \cite[(23), (25)]{bssschrod}, the latter estimate following from \cite[\S 6]{smithsogge06}. Strictly speaking, the power $\sigma$ appearing there is stated only for certain values of $p$, $q$. However, as motivated above, it also holds for the value of $\sigma(p,q)$ determined by~\eqref{sigmachoice} since an estimate of the form \eqref{Wdispersive} is established there.

\appendix
\section{Remarks on the proof of Theorem \ref{thm:nrgdk}}
Here we provide a brief outline of Ivanovici's proof of Theorem \ref{thm:nrgdk} and make some additional remarks.  As in \S2, attention will be restricted to Neumann conditions.  The estimates in \eqref{homognrgdk}, \eqref{inhomognrgdk} involve localization with respect to the Fourier variable dual to $t$, as opposed to the localization with respect to the spectrum of boundary Laplacian appearing in Ivanovici's work, but by the functional calculus these are essentially equivalent.  By a duality argument, it suffices to prove \eqref{inhomognrgdk}, leading one to consider solutions $w:\Omega \to \CC$ to the Helmholtz equation
\begin{equation*}
(\Delta-\l^2) w = \chi_j g, \qquad \prtl_\nu w\big|_{\prtl \Omega} = 0,
\end{equation*}
satisfying the (outgoing) Sommerfeld radiation condition
\begin{equation}\label{sommerfeld}
\lim_{r\to\infty}r^{\frac{n-1}{2}}\left(\prtl_r w - i\l w\right) =0, \qquad |x| =r.
\end{equation}
In \cite[\S2.3]{ivanoballs}, \cite[\S3.1]{ivanorevise} and the estimates in Theorem \ref{thm:nrgdk} are reduced to the bounds
\begin{equation}\label{helmholtzbounds}
\|\chi_j w\|_{L^2(\Omega)} \lesssim \l^{-1}2^{-\frac j2} \|\chi_j g\|_{L^2(\Omega)}.
\end{equation}
This reduction follows from taking the Fourier transform in the time variable (via a limiting procedure) and then observing standard results on existence and uniqueness of solutions to this equation.  For the latter, the interested reader can find a supplemental treatment in Theorems 4.37 and 4.38 in \cite{leis}, which works for homogeneous Neumann (and Dirichlet) conditions.

The bounds \eqref{helmholtzbounds} then follow by taking polar coordinates $(r,\omega) \in [1,\infty) \times \sph^{n-1}$ on $\Omega$ so that the Laplace operator takes the form
$$
\Delta = -\prtl_r^2 - \frac {n-1}r \prtl_r + \frac{1}{r^2}\Delta_{\sph^{n-1}},
$$
with $\Delta_{\sph^{n-1}}$ denoting the Laplace-Beltrami operator on $\sph^{n-1}$.  Let $\{\varphi_l\}_{1}^\infty$ to denote an orthonormal eigenbasis on $L^2(\sph^{n-1})$ satisfying $\Delta_{\sph^{n-1}}\varphi_l= \mu_l \varphi_l$.  Now write $w(r,\omega) = \sum_1^\infty w_l(r)\varphi_l(\omega)$ and $\chi_j g(r,\omega) = \sum_1^\infty g_l(r)\varphi_l(\omega)$.  Therefore, if we set $\nu_l = \left(\mu_l +\frac{(n-2)^2}{4}\right)^{\frac 12}$, then each $w_l(r)$ will satisfy \eqref{sommerfeld} and
\begin{equation}\label{lnueqn}
L_{\nu_l} w_l  := \left(-\prtl_r^2 - \frac{n-1}r \prtl_r + \frac{\nu_l^2}{r^2}-\l^2\right)w_l = g_l, \qquad
\prtl_r w_l(1)=0.
\end{equation}
Let $d\rho$ denote the measure $r^{n-1} dr$. By orthogonality, it suffices to see
\begin{equation}\label{modeineq}
\|w_l\|_{L^2(A_j,d\rho)} \lesssim \l^{-1}2^{-\frac j2} \|g_l\|_{L^2(A_j,d\rho)}, \qquad A_j := \{r: 1 \leq r \leq 1+ 2^{-j+2} \}
\end{equation}

Suppressing the $l$ in the notation, consider solutions $w$ to $L_\nu w(r)= g(r)$ where $\nu \geq 0$ and $\supp(g) \subset A_j$.  Let $H_\nu(z)$ denote the Hankel function of the first kind, order $\nu$.  As observed in \cite[(2.31)]{ivanorevise}, the Green's kernel for the problem in \eqref{lnueqn} satisfying the outgoing radiation condition can be written for $r \geq s \geq 1$ as
\begin{equation}\label{greenfcn}
G_{\nu,\l}(r,s)=\frac{\pi}{2i} (rs)^{1-\frac n2} \left(\overline{H}_\nu(\l s)-\frac{\l^{-1}(1-\frac n2)\overline{H}_\nu(\l)+\overline{H}_\nu'(\l)}{\l^{-1}(1-\frac n2)H_\nu(\l)+ H_\nu'(\l)}H_\nu(\l s) \right)H_\nu (\l r),
\end{equation}
and the remaining values are determined by symmetry $G_{\nu,\l}(r,s)= G_{\nu,\l}(s,r)$.  Therefore $w(r)$ is given by
$
w(r) = \int_{A_j} G_{\nu,\l}(r,s) g(s)s^{n-1}\,ds,
$
and the bounds \eqref{modeineq} follow from showing that $\|G_{\nu,\l}\|_{L^2(A_j\times A_j)} \lesssim \l^{-1}2^{-\frac j2}$.  When $\l \geq \nu$, this in turn follows from bounds on the $L^2$ norm of the dilated Hankel functions as the coefficient of $H_\nu(\l s)$ in \eqref{greenfcn} has modulus one.  The desired $L^2$ estimates on the $H_\nu(\l \cdot)$, then follow as in Propositions 2.5 and 2.6 in \cite{ivanorevise}.  Indeed, one can take $2^{-j} \approx \l^{-\alpha}$ in that proof and by suitable bounds on the Hankel functions, the implicit constants there will be independent of $j$.

Strictly speaking, in the case $\l < \nu$, the Neumann case requires some additional care when the ODE \eqref{lnueqn} transitions from elliptic to non-elliptic behavior.  This is because one needs to use the coefficient of the $H_\nu(\lambda s)$ in \eqref{greenfcn} to counterbalance the exponential growth of this function when $\l s$ decreases away from $\nu$. To this end, we observe bounds on Bessel functions $J_\nu(\nu z)$,  $Y_\nu(\nu z)$ of the first and second kinds (which satisfy $H_\nu(\nu z) = J_\nu(\nu z) + i Y_\nu(\nu z)$ when $z \in \RR$) and their derivatives.  When $0< z \leq 1-\nu^{-\frac 23}$, we have
\begin{align}
|Y_\nu(\nu z)| &\approx \frac{1}{\nu^{\frac 12}(1-z^2)^{\frac 14}}\exp\left(\frac 23 \nu \zeta^{\frac 32}\right),
& |J_\nu(\nu z)| &\approx \frac{1}{\nu^{\frac 12}(1-z^2)^{\frac 14}}
\exp\left(-\frac 23 \nu \zeta^{\frac 32}\right),\label{besselsmallz}\\
|Y_\nu'(\nu z)| &\approx \frac{(1-z^2)^{\frac 14}}{\nu^{\frac 12}z}\exp\left(\frac 23 \nu \zeta^{\frac
32}\right),  & |J_\nu'(\nu z)| &\approx \frac{(1-z^2)^{\frac 14}}{\nu^{\frac 12}z}
\exp\left(-\frac 23 \nu \zeta^{\frac 32}\right).\label{besselprimesmallz}
\end{align}
where $\zeta$ is the decreasing function defined by $\frac 23 \zeta^{\frac 32} = \int_z^1 \frac{\sqrt{1-t^2}}{t}\,dt$.  This is a consequence of results of Olver, which proves asymptotic bounds on Bessel functions which are uniform in $z$ and can be differentiated (see e.g. Theorem 3.1, Chapter 11 in \cite{olverbook}, the observations (10.18) and Ex. 10.1 in \S10 there, and combine these with typical asymptotics on Airy functions).  For the sake of completeness, we also state some uniform bounds his work yields on $H_\nu(\nu z)$ for $z \geq 1-\nu^{-\frac 23}$,
\begin{equation}\label{bessellargez}
|H_\nu(\nu z)| \approx \begin{cases}
\nu^{-\frac 13}, & z \in [1-\nu^{-\frac 23},1+\nu^{-\frac 23}],\\
\nu^{-\frac 12}(z^2-1)^{-\frac 14}, & z \in [ 1+\nu^{-\frac 23},\infty).
\end{cases}
\end{equation}
The bounds \eqref{besselsmallz}, \eqref{bessellargez} are slight variations on the ones appearing in \cite{ivanoballs}, \cite{ivanorevise}.

The main idea is that these estimates yield the pointwise bound
\begin{equation}\label{greenptwise}
|G_{\nu,\l}(r,s)|^2 \lesssim \frac{1}{\nu\sqrt{|(\l s/\nu)^2 -1| }\sqrt{|(\l r/\nu)^2 -1| }}.
\end{equation}
By symmetry it suffices to see this when $r \geq s \geq 1$.  It is illustrative to rewrite the kernel in \eqref{greenfcn} in terms of $H_\nu(z)$, $J_\nu(z)$ as
\begin{equation*}
\frac{\pi}{2i} (rs)^{1-\frac n2} \left(J_\nu(\l s) H_\nu (\l r) -\frac{\l^{-1}(1-\frac n2)J_\nu(\l)+J_\nu'(\l)}{\l^{-1}(1-\frac
n2)H_\nu(\l)+H_\nu'(\l)}H_\nu(\l s) H_\nu (\l r)\right).
\end{equation*}
Thus since $\zeta$ is decreasing, \eqref{besselsmallz} implies that $|J_\nu(\l s)H_\nu(\l r)|^2$ is bounded by the right hand side of \eqref{greenptwise}.  Similarly, the bounds \eqref{besselprimesmallz} show that the coefficient of $H_\nu(\l s) H_\nu (\l r)$ is exponentially small, enough to dominate the exponential growth of this function as $r \to 1^+$.  Given \eqref{greenptwise}, matters are reduced to seeing that
\begin{equation}\label{ajbound}
\frac{1}{\nu} \int_{A_j} \frac{ds}{\sqrt{|1-(\l s/\nu)^2|} } \lesssim \l^{-1}2^{-\frac j2},
\end{equation}
which is only a small variation on the estimates in the aforementioned propositions in \cite{ivanorevise} and also \cite[\S2.2]{ivanoballs}, the main idea being that $|1-z^2|^{-\frac 12}$ is locally integrable.

\end{document}